\documentclass[journal]{IEEEtran}

\usepackage{amsmath,amssymb,amsthm,mathrsfs,bbm}
\usepackage[colorlinks=true]{hyperref}
\usepackage{graphicx}
\usepackage{enumitem}

\newtheorem{theorem}{Theorem}
\newtheorem{definition}{Definition}
\newtheorem{proposition}{Proposition}
\newtheorem{lemma}{Lemma}
\newtheorem{corollary}{Corollary}
\newtheorem{claim}{Claim}
\theoremstyle{definition}\newtheorem{example}{Example}
\newtheorem{remark}{Remark}

\DeclareMathOperator*{\argmax}{arg\,max}
\newcommand{\bbar}[1]{\overline{#1}}
\newcommand{\ttilde}[1]{\widetilde{#1}}

\def\bs{\backslash}
\def\di{\displaystyle}
\def\eps{\varepsilon}
\def\t{\tau}
\def\d{\mathrm{d}}

\def\R{\mathbb{R}}
\def\N{\mathbb{N}}
\def\T{\mathbb{T}}

\def\NNN{\mathrm{N}}
\def\U{\mathrm{U}}
\def\B{\mathrm{B}}
\def\BB{\overline{\B}}
\def\C{\mathrm{C}}
\def\D{\mathrm{D}}
\def\P{\mathrm{P}}
\def\K{\mathrm{K}}
\def\AC{\mathrm{AC}}
\def\PC{\mathrm{PC}}
\def\L{\mathrm{L}}
\def\E{\mathrm{E}}

\def\proj{\mathrm{proj}}
\def\Pont{\mathrm{Pont}}

\def\Int{\mathrm{Int}}
\def\conv{\mathrm{conv}}

\def\UU{\mathcal{U}}
\def\ZZ{\mathcal{Z}}
\def\VV{\mathcal{V}}
\def\II{\mathcal{I}}
\def\PP{\mathcal{P}}
\def\TT{\mathcal{T}}
\def\NN{\mathcal{N}}
\def\BBB{\mathcal{B}}

\begin{document}

\title{Robustness under control sampling of reachability in fixed time for nonlinear control systems}

\author{Lo\"ic Bourdin and Emmanuel Tr\'elat
\thanks{L. Bourdin is with University of Limoges, XLIM Research Institute, F-87000 Limoges, France, loic.bourdin@unilim.fr.}
\thanks{E. Tr\'elat is with Sorbonne Universit\'e, Laboratoire Jacques-Louis Lions, F-75005 Paris, France, emmanuel.trelat@sorbonne-universite.fr.}
}

\maketitle

\begin{abstract}
Under a regularity assumption we prove that reachability in fixed time for nonlinear control systems is robust under control sampling.
\end{abstract}


\IEEEpeerreviewmaketitle

\section{Introduction and main result}\label{sec_intro}
Let $n$, $m \in\N^*$ and~$T > 0$ be fixed. In this work we consider the general nonlinear control system
\begin{equation}\label{contsys}\tag{CS}
\dot{x}(t) = f ( x(t),u(t),t ), \quad \text{a.e.\ } t \in [0,T],
\end{equation}
where the \textit{dynamics} $f : \R^n \times \R^m \times [0,T] \to \R^n$ is a continuous mapping, of class~$\C^1$ with respect to its first two variables.\footnote{This regularity assumption can be relaxed at several occasions in the paper (see Remark~\ref{rem_relaxregularityonf} for details).} We say that a pair~$(x,u)$ is a solution to~\eqref{contsys} if~$x \in \AC([0,T],\R^n)$ is an absolutely continuous function (called \textit{state} or \textit{trajectory}) and~$u\in \L^\infty([0,T],\R^m)$ is an essentially bounded measurable function (called \textit{control}) such that~$\dot{x}(t) = f ( x(t),u(t),t )$ for a.e.\ $t \in [0,T]$. Throughout the paper we fix a starting point~$x^0\in\R^n$ and a nonempty subset~$\U$ of~$\R^m$ standing for the set of control constraints. We say that a target point~$x^1\in\R^n$ is~\emph{$\L^\infty_\U$-reachable in time~$T$ from~$x^0$} if there exists a solution~$(x,u)$ to~\eqref{contsys}, with~$u\in \L^\infty([0,T],\U)$, such that~$x(0)=x^0$ and~$x(T)=x^1$.

We now \textit{sample} the control~$u$ over the time interval~$[0,T]$: given a partition~$\T = \{ t_i \}_{i=0,\ldots,N}$ of~$[0,T]$, consisting of real numbers satisfying $ 0 = t_0 < t_1 < \ldots < t_{N-1} < t_N = T $, for some~$N \in \N^*$, we consider the set~$\PC^\T([0,T],\R^m)$ of all possible piecewise constant functions~$u:[0,T]\rightarrow \R^m$ satisfying~$u(t)=u_i$, for some~$u_i \in \R^m$, for every~$t\in[t_i,t_{i+1})$ and every~$i \in \{ 0,\ldots,N-1 \}$. We denote by~$ \Vert \T \Vert = \max_{i=0,\ldots,N-1} \vert t_{i+1} - t_i \vert $ the norm of the partition. We say that a target point~$x^1\in\R^n$ is~\emph{$\PC^\T_\U$-reachable in time~$T$ from~$x^0$} if there exists a solution~$(x,u)$ to~\eqref{contsys}, with~$u\in\PC^\T([0,T],\U)$, such that~$x(0)=x^0$ and~$x(T)=x^1$.

In this paper we investigate the following question: assuming that~a target point~$x^1\in\R^n$ is~$\L^\infty_\U$-reachable in time~$T$ from~$x^0$ and given a partition~$\T$ of~$[0,T]$, is the point~$x^1$ also~$\PC^\T_\U$-reachable in time~$T$ from~$x^0$? In other words, how robust is reachability in fixed time under control sampling? Without any specific assumption, even for small values of~$\Vert \T \Vert$, in general $x^1$ fails to be $\PC^\T_\U$-reachable in time $T$ from $x^0$, as shown in the following example.

\begin{example}\label{example1}
Take~$T=n=m=1$, $\U = \R$ and~$f(x,u,t) = 1+(u-t)^2$ for all~$(x,u,t) \in \R \times \R \times [0,T]$. The target point~$x^1 = 1$ is $\L^\infty_\U$-reachable in time~$T$ from the starting point~$x^0 = 0$ with the control~$u(t)=t$ for a.e.\ $t \in [0,T]$. However there is no other control steering the control system from~$x^0$ to~$x^1$ in time~$T$. Therefore, given any partition~$\T$ of~$[0,T]$, even with a small value of~$\Vert \T \Vert$, the target point~$x^1$ is not~$\PC^\T_\U$-reachable in time~$T$ from~$x^0$.
\end{example}

Our main result is the following.

\begin{theorem}\label{thmmain}
Assume that~$\U$ is convex and let~$x^1 \in \R^n$ be a target point that is~$\L^\infty_\U$-reachable in time~$T$ from $x^0$ with a control~$u \in \L^\infty([0,T],\U)$. If $u$ is weakly~$\U$-regular, then there exists a threshold~$\delta > 0$ such that~$x^1$ is~$\PC^\T_\U$-reachable in time~$T$ from~$x^0$ for any partition~$\T$ of~$[0,T]$ satisfying~$\Vert \T \Vert \leq \delta$.
\end{theorem}

The key concept of \emph{weakly $\U$-regular control} is defined, commented and characterized in Section~\ref{sec_recap}, in relation with local reachability results. Theorem~\ref{thmmain} is discussed in detail in Section~\ref{sec_comments}. In particular we emphasize here that the convexity assumption made on~$\U$ and the $\C^1$ smoothness assumption made on~$f$ can both be relaxed (see Remarks~\ref{remequalityinequalityU}, \ref{remequalityinequalityU2} and~\ref{rem_relaxregularityonf} for details). All proofs are done in Sections~\ref{secproofs1} and~\ref{secproofs2}.

\section{Recap on local reachability results}\label{sec_recap}

This section gathers in a concise way a number of local reachability results, helpful for various purposes all along this paper. Most of these results are well known in the literature (see, e.g.,~\cite{agrachev2004, BonnardChyba, Coron_book, LeeMarkus,pontryagin1962,sontag1998,trelat2005} and references therein), while others are less known or even new. 

In Section~\ref{sec_regwithout} we deal with the unconstrained control case (i.e., when~$\U=\R^m$), recalling how the implicit function theorem can provide local reachability results thanks to the notion of \emph{strongly regular control}. In Section~\ref{sec_regwith} we show how to extend this approach under convex control constraints (i.e., when~$\U$ is a convex subset of~$\R^m$), thanks to the notion of \emph{strongly~$\U$-regular control} and to a \emph{conic} version of the implicit function theorem. In Section~\ref{sec_wregwith} we treat the general control constraints case (i.e., when~$\U$ is a general subset of~$\R^m$), thanks to the notion of \emph{weakly~$\U$-regular control} and using needle-like variations. These different notions lead to distinct results, that we comment further in Section~\ref{sec_further_relationships}.

We first recall some basic facts and terminology. A control~$u \in \L^\infty([0,T],\R^m)$ is said to be \emph{admissible} when there exists~$x \in \AC([0,T],\R^n)$, starting at~$x(0)=x^0$, such that~$(x,u)$ is a solution to~\eqref{contsys}. In that case the trajectory~$x$ is unique and will be denoted by~$x_u$. The set~$\UU$ of all admissible controls is an open subset of~$\L^\infty([0,T],\R^m)$ and the \emph{end-point mapping}~$\E : \UU \to \R^n$ is the~$\C^1$ mapping defined by~$\E(u) = x_u(T)$ for every~$u \in \UU$. 
Therefore a target point~$x^1\in\R^n$ is~$\L^\infty_\U$-reachable in time~$T$ from~$x^0$ if and only if~$x^1 $ belongs to the~$\L^\infty_\U$-\textit{accessible set} given by~$\E ( \UU \cap \L^\infty([0,T],\U) )$. Reachability in time~$T$ from~$x^0$ is thus related to a surjectivity property of~$\E$.

\subsection{Without control constraint}\label{sec_regwithout}
All results in this section are classical (see, e.g., \cite{agrachev2004,BonnardChyba, Coron_book,trelat2005}). When~$\U=\R^m$, i.e., when there is no control constraint, some conditions ensuring surjectivity of~$\E$ are well known. For instance, when the control system~\eqref{contsys} is linear and autonomous, i.e., $f(x,u,t)=Ax+Bu+g(t)$ for all~$(x,u,t) \in \R^n \times \R^m \times [0,T]$, where~$A \in \R^{n \times n}$ and~$B \in \R^{n \times m}$ are constant matrices and~$g \in \C([0,T],\R^n)$ is a continuous function, we have~$\UU = \L^\infty([0,T],\R^m)$, and~$\E$ is surjective if and only if the pair~$(A,B)$ satisfies the classical Kalman condition. For a general nonlinear control system~\eqref{contsys}, \emph{global} surjectivity of~$\E$ cannot be ensured in general. But, thanks to the implicit function theorem, \emph{local} surjectivity can be established (see Proposition~\ref{prop_reginterior} below, proved in Section~\ref{secproof_reginterior}).



\begin{definition}[strongly\footnote{With respect to the existing literature, we add the word ``strongly", in contrast to the notion of ``weakly" regular control defined in Section~\ref{sec_wregwith}.} regular control]\label{def_reg}
A control~$u \in \UU$ is said to be \emph{strongly regular} if the Fr\'echet differential~$\D\E(u):\L^\infty([0,T],\R^m)\rightarrow\R^n$ is surjective, i.e.,~$\mathrm{Ran}(\D\E(u))=\R^n$. A control~$u \in \UU$ is said to be \emph{weakly singular} if it is not strongly regular, i.e., $\mathrm{Ran}(\D\E(u))$ is a proper subspace of~$\R^n$. 
\end{definition}

\begin{proposition}\label{prop_reginterior}
If a control~$u \in \UU$ is strongly regular, then there exist an open neighborhood~$\VV$ of~$x_u(T)$ and a~mapping~$V : \VV \to \UU$ of class $\C^1$ satisfying~$V(x_u(T)) = u$ and~$\E ( V(z) ) = z$ for every~$z \in \VV$. In particular, any point of~$\VV$ is $\L^\infty_{\R^m}$-reachable in time~$T$ from~$x^0$, and thus~$x_u(T)$ belongs to the interior of the~$\L^\infty_{\R^m}$-accessible set.
\end{proposition}

A Hamiltonian characterization of weakly singular controls (recalled in Proposition~\ref{prop_charactsing} further) can be derived from the expression of the Fr\'echet differential of~$\E$ given by
\begin{equation}\label{eq1}
 \D \E (u) \cdot v =  w^u_v (T)
\end{equation}
for every~$u \in \UU$ and every~$v \in \L^\infty([0,T],\R^m)$, where~$w^u_v \in \AC([0,T],\R^n)$ is the unique solution to
\begin{equation*}
\left\lbrace
\begin{array}{l}
\dot{w}(t) =  \nabla_x f(x_u(t),u(t),t) w(t) \\
\qquad\qquad\qquad + \nabla_u f(x_u(t),u(t),t) v(t), \quad \text{a.e.\ } t \in [0,T] , \\
w(0) =  0_{\R^n} .
\end{array}
\right.
\end{equation*} 
The \emph{Hamiltonian} associated to~\eqref{contsys} is the function~$H : \R^n \times \R^m \times \R^n \times [0,T] \to \R$ defined by
$$ 
H(x,u,p,t) = \langle p,f(x,u,t) \rangle_{\R^n} 
$$
for all~$(x,u,p,t) \in \R^n \times \R^m \times \R^n \times [0,T]$, where~$\langle \cdot , \cdot \rangle_{\R^n}$ is the Euclidean scalar product in~$\R^n$. 

\begin{definition}[weak extremal lift]\label{def_weakextremallift}
A \emph{weak extremal lift} of a pair~$(x_u,u)$, where~$u \in \UU$, is a triple~$(x_u,u,p)$ where~$p \in \AC( [0,T] , \R^n )$ (called adjoint vector) is a solution to the (linear) \emph{adjoint equation}
\begin{equation}\tag{AE}\label{eqAE}
\dot{p}(t) = - \nabla_x H(x_u(t),u(t),p(t),t)  
\end{equation}
for a.e.\ $t \in [0,T]$, satisfying the \emph{null Hamiltonian gradient condition}
\begin{equation}\tag{NHG}\label{eqNHG}
\nabla_u H(x_u(t),u(t),p(t),t) =0_{\R^m} 
\end{equation}
for a.e.\ $t \in [0,T]$. The weak extremal lift~$(x_u,u,p)$ is said to be \emph{nontrivial} if~$p$ is nontrivial.
\end{definition}


\begin{proposition}\label{prop_charactsing}
A control~$u \in \UU$ is weakly singular if and only if the pair~$(x_u,u)$ admits a nontrivial weak extremal lift.
\end{proposition}

While the definition of weakly singular control is quite abstract, the above classical Hamiltonian characterization (proved in Section~\ref{secproof_charactsing}) is practical. For example one can easily prove that the control in Example~\ref{example1} is weakly singular.

As a consequence of Propositions~\ref{prop_reginterior} and~\ref{prop_charactsing}, if a pair~$(x_u,u)$, for some~$u \in \UU$, has no nontrivial weak extremal lift, then any point of an open neighborhood of~$x_u(T)$ is~$\L^\infty_{\R^m}$-reachable in time~$T$ from~$x^0$. Note that the contrapositive statement corresponds to a weak version of the geometric Pontryagin maximum principle: if~$x_u(T)$, for some~$u \in \UU$, belongs to the boundary of the~$\L^\infty_{\R^m}$-accessible set, then the pair~$(x_u,u)$ admits a nontrivial weak extremal lift.

\subsection{With convex control constraints}\label{sec_regwith}
When~$\U$ is a proper subset of~$\R^m$, i.e., when there are control constraints, reachability properties are more difficult to establish in general. We refer the reader to~\cite{Brammer} for conic-type conditions for autonomous linear control systems, to~\cite{BoscainPiccoli, KrenerSchattler} for single-input control-affine systems in dimensions~$2$ and~$3$, and to~\cite{BianWebb,SussmannJurdjevic} for more general systems. 


In Section~\ref{sec_regwithout} we have recalled that, in the absence of control constraint, local reachability can be ensured thanks to the classical implicit function theorem, by assuming that~$\D\E(u)$ is surjective for some~$u \in \UU$. When there are control constraints, a powerful approach is to use \emph{constrained versions} of the implicit function theorem (as in \cite{antoine1990, BianWebb, Klamka}). When~$\U$ is convex, the required hypothesis for a control~$u \in \UU \cap \L^\infty([0,T],\U)$ is a \textit{conic} surjectivity assumption made on~$\D \E(u)$ as follows.

\begin{definition}[strongly $\U$-regular control]\label{def_Ureg}
Assume that~$\U$ is convex. A control~$u \in \UU \cap \L^\infty([0,T],\U)$ is said to be~\emph{strongly $\U$-regular} if~$\D \E(u)(\TT_{\L^\infty_\U}[u]) = \R^n$, where~$\TT_{\L^\infty_\U}[u]$ is the (convex) \emph{tangent cone} to $\L^\infty([0,T],\U)$ at $u$ defined by
\begin{equation*}
\TT_{\L^\infty_\U}[u] = \R_+ ( \L^\infty([0,T],\U) - u ) . 
\end{equation*}
The control~$u$ is said to be \emph{weakly $\U$-singular} when it is not strongly $\U$-regular, i.e., $\D \E(u)(\TT_{\L^\infty_\U}[u])$ is a proper convex subcone of~$\R^n$.
\end{definition}

\begin{proposition}\label{prop_Ureginterior}
Assume that~$\U$ is convex. If a control~$u \in \UU \cap \L^\infty([0,T],\U)$ is strongly $\U$-regular, then there exist an open neighborhood~$\VV $ of~$x_u(T)$ and a continuous mapping~$V : \VV \to \UU \cap \L^\infty([0,T],\U)$ satisfying~$V(x_u(T)) = u$ and~$\E ( V(z) ) = z$ for every~$z \in \VV$. In particular, any point in~$\VV$ is $\L^\infty_{\U}$-reachable in time~$T$ from~$x^0$, and thus~$x_u(T)$ belongs to the interior of the~$\L^\infty_{\U}$-accessible set.
\end{proposition}

The proof of Proposition~\ref{prop_Ureginterior}, based on the conic implicit function theorem~\cite[Theorem~1]{antoine1990}, is provided in Section~\ref{secproof_Ureginterior}. 

Similar results to Proposition~\ref{prop_Ureginterior} are known in the literature. For example it echoes results obtained in~\cite{BianWebb,Klamka} in which the sufficient condition is settled as a constrained controllability property of the linearized control system. Such a condition is however not easy to check in practice.

As in the unconstrained control case (Section~\ref{sec_regwithout}), we next provide a practical Hamiltonian characterization of weakly $\U$-singular controls.

\begin{definition}[weak $\U$-extremal lift]\label{def_Uextremallift}
Assume that~$\U$ is convex. A~\emph{weak $\U$-extremal lift} of a pair~$(x_u,u)$, where~$u \in \UU \cap \L^\infty([0,T],\U)$, is a triple~$(x_u,u,p)$ where~$p \in \AC([0,T],\R^n)$ (called adjoint vector) is a solution to the adjoint equation~\eqref{eqAE} satisfying the \emph{Hamiltonian gradient condition}
\begin{equation}\tag{HG}\label{eqHG}
\nabla_u H(x_u(t),u(t),p(t),t) \in \NN_\U [u(t)] 
\end{equation}
for a.e.\ $t \in [0,T]$, where 
$$ 
\NN_\U[u(t)] = \{ \vartheta \in \R^m \mid \forall \omega \in \U, \; \langle \vartheta , \omega - u(t) \rangle_{\R^m} \leq 0 \}
$$
is the \emph{normal cone} to~$\U$ at~$u(t)$. The weak $\U$-extremal lift~$(x_u,u,p)$ is said to be \emph{nontrivial} if~$p$ is nontrivial.
\end{definition}

\begin{proposition}\label{prop_charactUsing}
Assume that~$\U$ is convex. A control~$u \in \UU \cap \L^\infty([0,T],\U)$ is weakly $\U$-singular if and only if the pair~$(x_u,u)$ admits a nontrivial weak $\U$-extremal lift.
\end{proposition}

The proof of Proposition~\ref{prop_charactUsing}, using in particular needle-like variations (recalled in Section~\ref{sec_wregwith}), is done in Section~\ref{secproof_charactUsing}. 

As a consequence of Propositions~\ref{prop_Ureginterior} and~\ref{prop_charactUsing}, when~$\U$ is convex, if a pair~$(x_u,u)$, for some~$u \in \UU \cap \L^\infty([0,T],\U)$, has no nontrivial weak~$\U$-extremal lift, then any point of an open neighborhood of~$x_u(T)$ is~$\L^\infty_{\U}$-reachable in time~$T$ from~$x^0$. The contrapositive statement corresponds to a weak version of the geometric Pontryagin maximum principle in the presence of convex control constraints: when~$\U$ is convex, if the point~$x_u(T)$, for some~$u \in \UU \cap \L^\infty([0,T],\U)$, belongs to the boundary of the~$\L^\infty_\U$-accessible set, then the pair~$(x_u,u)$ admits a nontrivial weak $\U$-extremal lift.

\begin{remark}\label{rem1}
When~$\U$ is convex, it is clear that, if a control~$u \in \UU \cap \L^\infty([0,T],\U)$ is strongly~$\U$-regular, then it is strongly regular. The converse is not true in general, as shown in the next example, but is true when~$u$ takes its values in the interior of~$\U$ (see Proposition~\ref{propint} further), in particular when~$\U=\R^m$.
\end{remark}

\begin{example}\label{example2}
Take~$T=n=m=1$, $\U = [-1,1]$ and~$f(x,u,t) = u$ for all~$(x,u,t) \in \R \times \R \times [0,T]$. From the Hamiltonian characterizations, the constant control~$u\equiv 1$ is strongly regular and weakly $\U$-singular.
\end{example}

\begin{remark}
Note that the conclusions of Propositions~\ref{prop_reginterior} and~\ref{prop_Ureginterior} are distinct: the local right-inverse mapping~$V$ is of class $\C^1$ in Proposition \ref{prop_reginterior} (in the unconstrained control case), while it is (only) continuous in Proposition \ref{prop_Ureginterior} (in the convex control constraints case). In the latter, obtaining $\C^1$ smoothness is an open question. Indeed, in all references on \emph{constrained} implicit function theorems we found (such as~\cite{antoine1990}), the continuity of the local right-inverse mapping is established, but obtaining~$\C^1$ smoothness does not seem to be an easy issue.
\end{remark}

\subsection{With general control constraints}\label{sec_wregwith}

When~$\U$ is convex and for a given control~$u \in \UU \cap \L^\infty([0,T],\U)$, the set~$\D \E ( u ) ( \TT_{\L^\infty_\U}[u])$ consists of all elements~$w^u_v(T)$ (called the \textit{weak $\U$-variation vectors} associated with~$u$) generated by conic~$\L^\infty$-perturbations~$u + \alpha v$ of the control~$u$, where~$v \in \TT_{\L^\infty_\U}[u]$ and~$\alpha \geq 0$, in the sense that
\begin{equation*}
w^u_v (T) = \lim\limits_{\alpha \to 0^+} \dfrac{\E (u+\alpha v)-\E(u)}{\alpha} = \D \E (u) \cdot v.
\end{equation*}
The set~$\D \E ( u ) ( \TT_{\L^\infty_\U}[u])$ can be seen as a first-order conic convex approximation of the $\L^\infty_\U$-accessible set at~$x_u(T)$.

In the general control constraints case where $\U$ is not assumed to be convex, we can use sophisticated~$\L^1$-perturbations, well known in the literature as \textit{needle-like variations}. Precisely a needle-like variation $u^\alpha_{(\t,\omega)} \in \L^\infty([0,T],\U)$ of a given control~$u \in \UU \cap \L^\infty([0,T],\U)$ is defined by
\begin{equation}\label{eq_singleneedle}
u^\alpha_{(\t,\omega)}(t) = \left\lbrace \begin{array}{ll}
\omega & \text{along } [\t,\t+\alpha), \\[3pt]
u(t) & \text{elsewhere,} 
\end{array}
\right.
\end{equation}
for a.e.\ $t \in [0,T]$ and for~$ \alpha \geq 0$, with~$(\t,\omega) \in \mathcal{L}( f_u) \times \U $, where~$\mathcal{L}( f_u)$ stands for the full-measure set of all Lebesgue points in $[0,T)$ of the essentially bounded measurable function~$f_u = f(x_u,u,\cdot)$. In that framework, it is well known that~$u^\alpha_{(\t,\omega)}$ belongs to~$ \UU$ for sufficiently small~$\alpha \geq 0$ and that
\begin{equation}\label{eq2}
\lim\limits_{\alpha \to 0^+} \dfrac{\E(u^\alpha_{(\t,\omega)})-\E(u)}{\alpha} = w^u_{(\t,\omega)} (T),
\end{equation}
where~$w^u_{(\t,\omega)} \in \AC([\t,T],\R^n)$ is the unique solution to
$$
\left\lbrace
\begin{array}{l}
\dot{w}(t) =  \nabla_x f(x_u(t),u(t),t) w(t) , \quad \text{a.e.\ } t \in [\t,T] , \\
w(\t) =  f(x_u(\t),\omega,\t) - f(x_u(\t),u(\t),\t) .
\end{array}
\right.
$$
The elements~$w^u_{(\t,\omega)} (T)$, with~$(\t,\omega) \in \mathcal{L}( f_u) \times \U $, are called the \textit{strong $\U$-variation vectors} associated with~$u$. 

\begin{definition}[$\U$-Pontryagin cone]\label{def_pontcone}
The \emph{$\U$-Pontryagin cone} of a control~$u \in \UU \cap \L^\infty([0,T],\U)$, denoted by $\Pont_\U[u]$, is the smallest convex cone containing all strong $\U$-variation vectors associated with~$u$.\footnote{In the literature, usually the $\U$-Pontryagin cone of a control~$u \in \UU \cap \L^\infty([0,T],\U)$ is defined as the smallest \textit{closed} convex cone containing all strong~$\U$-variation vectors associated with~$u$ (see, e.g., \cite{LeeMarkus}). As explained in Remark~\ref{reminclusion}, considering the closure (or not) has no impact on the notions and results presented in this paper. Nevertheless we emphasize that the multiple needle-like variations of the control~$u$ (see Section~\ref{secproof_wUreginterior}) generate (only) the~$\U$-Pontryagin cone of~$u$ as defined in Definition~\ref{def_pontcone} (i.e., without closure).}
\end{definition}

A strong $\U$-variation vector associated with a control~$u \in \UU \cap \L^\infty([0,T],\U)$ is generated in~\eqref{eq2} by using a \textit{single} needle-like variation~\eqref{eq_singleneedle}: this is standard in the literature. What is less standard is that, actually, the $\U$-Pontryagin cone, which consists of all conic convex combinations of strong~$\U$-variation vectors associated with~$u$, can be generated by using \textit{multiple} needle-like variations (see Section~\ref{secproof_wUreginterior}). Hence the set~$\Pont_\U[u]$ can be seen as a first-order conic convex approximation of the $\L^\infty_\U$-accessible set at~$x_u(T)$. Note that, when~$\U$ is convex, it is larger than~$\D \E ( u ) ( \TT_{\L^\infty_\U}[u])$ (in the sense of Remark~\ref{reminclusion}) which leads to the following weakened notion of~$\U$-regularity.

\begin{definition}[weakly $\U$-regular control]\label{def_wUreg}
A control~$u \in \UU \cap \L^\infty([0,T],\U)$ is said to be~\emph{weakly $\U$-regular} if~$\Pont_\U[u] = \R^n$. The control~$u$ is said to be~\emph{strongly $\U$-singular} when it is not weakly~$\U$-regular, i.e., $\Pont_\U[u]$ is a proper convex subcone of~$\R^n$.
\end{definition}

Although the $\U$-Pontryagin cone of a control~$u \in \UU \cap \L^\infty([0,T],\U)$ cannot be written as the range of a differential~$\D \E (u)$ taken in an appropriate sense (see Remark~\ref{rem_conenonimage}), the next proposition can be obtained by applying the conic implicit function theorem~\cite[Theorem~1]{antoine1990} to a restriction of~$\E$ to a multiple needle-like variation (see the proof in Section~\ref{secproof_wUreginterior}).

\begin{proposition}\label{prop_wUreginterior}
If a control~$u \in \UU \cap \L^\infty([0,T],\U)$ is weakly~$\U$-regular, then there exist an open neighborhood~$\VV $ of~$x_u(T)$ and a mapping~$V : \VV \to \UU \cap \L^\infty([0,T],\U)$, that is continuous when endowing the codomain with the~$\L^1$-metric, satisfying~$V(x_u(T)) = u$ and~$\E ( V(z) ) = z$ for all~$z \in \VV$. In particular any point in~$\VV$ is $\L^\infty_{\U}$-reachable in time~$T$ from~$x^0$, thus~$x_u(T)$ belongs to the interior of the~$\L^\infty_{\U}$-accessible set.
\end{proposition}

Like in Sections~\ref{sec_regwithout} and~\ref{sec_regwith}, we next provide a Hamiltonian characterization of strongly~$\U$-singular controls (see Proposition~\ref{prop_charactsUsing} below, proved in Section~\ref{secproof_charactsUsing}).

\begin{definition}[strong $\U$-extremal lift]\label{def_sUextremallift}
A~\emph{strong $\U$-extremal lift} of a pair~$(x_u,u)$, where~$u \in \UU \cap \L^\infty([0,T],\U)$, is a triple~$(x_u,u,p)$ where~$p \in \AC([0,T],\R^n)$ (called adjoint vector) is a solution to the adjoint equation~\eqref{eqAE} satisfying the \emph{Hamiltonian maximization condition}
\begin{equation}\label{eqHM}\tag{HM}
u(t) \in \argmax_{\omega \in \U} H(x_u(t),\omega,p(t),t)
\end{equation}
for a.e.\ $t \in [0,T]$. The strong~$\U$-extremal lift~$(x_u,u,p)$ is said to be \emph{nontrivial} if~$p$ is nontrivial.
\end{definition}

\begin{proposition}\label{prop_charactsUsing}
A control~$u \in \UU \cap \L^\infty([0,T],\U)$ is strongly~$\U$-singular if and only if the pair~$(x_u,u)$ admits a nontrivial strong~$\U$-extremal lift.
\end{proposition}

From Propositions~\ref{prop_wUreginterior} and~\ref{prop_charactsUsing}, if a pair~$(x_u,u)$, where~$u \in \UU \cap \L^\infty([0,T],\U)$, has no nontrivial strong~$\U$-extremal lift, then any point of an open neighborhood of~$x_u(T)$ is~$\L^\infty_{\U}$-reachable in time~$T$ from~$x^0$. The contrapositive  statement coincides exactly with the well known geometric Pontryagin maximum principle: if~$x_u(T)$, for some~$u \in \UU \cap \L^\infty([0,T],\U)$, belongs to the boundary of the~$\L^\infty_\U$-accessible set, then the pair~$(x_u,u)$ admits a nontrivial strong $\U$-extremal lift. 

\begin{remark}\label{reminclusion}
Let~$u \in \UU \cap \L^\infty([0,T],\U)$. Since~$\Pont_\U[u]$ is convex, we have $\Pont_\U[u] = \R^n$ if and only if its closure satisfies~$\mathrm{Clos}(\Pont_\U[u])= \R^n$. When~$\U$ is convex, we have $ \D \E ( u ) ( \TT_{\L^\infty_\U}[u]) \subset \mathrm{Clos}(\Pont_\U[u])$ and thus, if~$u$ is strongly~$\U$-regular, then it is weakly $\U$-regular.\footnote{This fact can also be derived from the Hamiltonian characterizations.} The converse is not true in general, as shown in the following two examples.
\end{remark}

\begin{example}\label{example3}
Take~$T=n=m=1$, $\U = [-1,1]$ and~$f(x,u,t) = u^3$ for all~$(x,u,t) \in \R \times \R \times [0,T]$. From the Hamiltonian characterizations, the constant control~$u \equiv 0$ is weakly~$\U$-regular and weakly $\U$-singular.
\end{example}

\begin{example}\label{example3bis}
Take~$T=n=1$, $m=2$, $\U = [-1,1]^2$ and~$f(x,(u_1,u_2),t) = u_1 u_2$ for all~$(x,(u_1,u_2),t) \in \R \times \R^2 \times [0,T]$. From the Hamiltonian characterizations, the constant control~$u \equiv 0_{\R^2}$ is weakly~$\U$-regular and weakly $\U$-singular.
\end{example}

\begin{remark}\label{rem56}
No relationship can be established between strong regularity and weak~$\U$-regularity in general. One can check that Example~\ref{example2} provides a control that is strongly regular and strongly~$\U$-singular, and that Example~\ref{example3} provides a control that is weakly singular and weakly~$\U$-regular. We refer to Propositions~\ref{propint} and~\ref{proplinear} further for relationships in special cases.
\end{remark}

\begin{remark}\label{rem13}
It follows from the Hamiltonian characterization that, if a control~$u \in \UU \cap \L^\infty([0,T],\U)$ is strongly~$\U$-singular on the interval~$[0,T]$ with the starting point~$x^0$, then it is also strongly~$\U$-singular on any subinterval~$[\t^0,\t^1] \subset [0,T]$ of nonempty interior with the starting point~$x_u(\t^0)$. When~$\U$ is convex, the same assertion is true when replacing ``strongly~$\U$-singular" with ``weakly $\U$-singular".
\end{remark}

\begin{remark}
Note that the conclusions of Propositions~\ref{prop_Ureginterior} and~\ref{prop_wUreginterior} are distinct. In Proposition~\ref{prop_Ureginterior}, when~$\U$ is convex and the control~$u \in \UU \cap \L^\infty([0,T],\U)$ is strongly~$\U$-regular, the controls allowing to reach an open neighborhood of~$x_u(T)$ can be chosen close to~$u$ in~$\L^\infty$-topology. In Proposition~\ref{prop_wUreginterior}, when the control~$u \in \UU \cap \L^\infty([0,T],\U)$ is (only) weakly~$\U$-regular, closedness is obtained in the weaker $\L^1$-topology (this is because needle-like variations are~$\L^1$-perturbations). There are similar subtleties in Section~\ref{sec_comments} due to the fact that piecewise constant functions are dense in~$\L^\infty([0,T],\R^m)$ when endowed with the~$\L^1$-norm (but not with the natural~$\L^\infty$-norm).
\end{remark}

\subsection{Additional comments and results}\label{sec_further_relationships}

The next proposition, which seems to be new, follows straightforwardly from the Hamiltonian characterizations and from the fact that, when~$\U$ is convex, the normal cone to~$\U$ at any interior point of~$\U$ is reduced to~$\{ 0_{\R^m} \}$.

\begin{proposition}\label{propint}
Let~$u \in \UU \cap \L^\infty([0,T],\Int(\U))$, where~$\Int(\U)$ is the interior of~$\U$. 
\begin{enumerate}[label=\rm{(\roman*)}]
\item If~$u$ is strongly regular, then~$u$ is weakly~$\U$-regular. The converse is not true in general (see Example~\ref{example3} and Remark~\ref{rem56}).
\item When~$\U$ is convex,~$u$ is strongly regular if and only if~$u$ is strongly~$\U$-regular.\footnote{This fact is obvious when~$u$ belongs to~$\Int(\L^\infty([0,T],\U))$ since then~$\TT_{\L^\infty_\U}[u] = \L^\infty([0,T],\R^m)$. However note that the inclusion~$\Int(\L^\infty([0,T],\U)) \subset \L^\infty([0,T],\Int(\U))$ may be strict (for a counterexample, take~$T=m=1$, $\U=[0,1]$ and~$u(t)=t$ for a.e.\ $t \in [0,T]$).}
\end{enumerate}
\end{proposition}

\begin{remark}\label{rem14}
By Remark~\ref{rem13} and Proposition~\ref{propint}, if a control~$u \in \UU \cap \L^\infty([0,T],\U)$ takes its values in $\Int(\U)$ along a subinterval~$[\t^0,\t^1]\subset [0,T]$ of nonempty interior on which it is moreover strongly regular with the starting point~$x_u(\t^0)$, then~$u$ is weakly~$\U$-regular (and even strongly~$\U$-regular if~$\U$ is convex) on~$[0,T]$ with the starting point~$x^0$.
\end{remark}

The control system~\eqref{contsys} is said to be \textit{control-affine} when~$f(x,u,t)=g(x,t)+B(x,t)u$ for all~$(x,u,t) \in \R^n \times \R^m \times [0,T]$, where~$g : \R^n \times [0,T] \to \R^n$ and~$B  : \R^n \times [0,T] \to \R^{n \times m}$ are continuous mappings, of class~$\C^1$ with respect to their first variable. In that context we have
$$ 
H(x,\omega,p,t) - H(x,u,p,t) = \langle \nabla_u H (x,u,p,t) , \omega - u \rangle_{\R^m}
$$
for all~$(x,u,\omega,p,t) \in \R^n \times \R^m \times \R^m \times \R^n \times [0,T]$ and the next proposition follows straightforwardly.

\begin{proposition}\label{proplinear}
Assume that the control system~\eqref{contsys} is control-affine and let~$u \in \UU \cap \L^\infty([0,T],\U)$. 
\begin{enumerate}[label=\rm{(\roman*)}]
\item If~$u$ is strongly~$\mathrm{conv}(\U)$-regular, where~$\conv(\U)$ is the convex hull of~$\U$, then~$u$ is weakly~$\U$-regular.
\item If~$u$ is weakly~$\U$-regular, then~$u$ is strongly regular. The converse is not true in general (see Example~\ref{example2} and Remark~\ref{rem56}).
\item When~$\U$ is convex,~$u$ is weakly~$\U$-regular if and only if~$u$ is strongly~$\U$-regular.
\end{enumerate}
\end{proposition}

As a particular case of control-affine system, the control system~\eqref{contsys} is said to be \textit{linear} when~$f(x,u,t)=A(t)x+B(t)u +g(t)$ for all~$(x,u,t) \in \R^n \times \R^m \times [0,T]$, where~$A \in \C([0,T], \R^{n \times n})$, $B \in \C([0,T], \R^{n \times m})$ and~$g \in \C([0,T], \R^n)$ are continuous functions. In that context~$\UU = \L^\infty([0,T],\R^m)$ and~$\E$ is affine. An example given in Appendix~\ref{appexamplegrasse} shows that the converse of the geometric Pontryagin maximum principle stated at the end of Section~\ref{sec_wregwith} is not true in general.\footnote{Since~$\U=\R^m$ in that example, it also shows that the converses of the weak versions of the geometric Pontryagin maximum principle stated at the end of Sections~\ref{sec_regwithout} and~\ref{sec_regwith} are also not true in general.} However, for linear control systems, the converse is true, as stated in the next proposition (proved in Section~\ref{secproof_proplinear2}).

\begin{proposition}\label{proplinear2a2b}
Assume that the control system~\eqref{contsys} is linear and let~$u \in \L^\infty([0,T],\U)$. Then $x_u(T)$ belongs to the interior of the~$\L^\infty_\U$-accessible set if and only if $u$ is weakly~$\U$-regular.
\end{proposition}

\begin{remark}\label{rem97}
Assume that the control system~\eqref{contsys} is linear and autonomous (i.e., $A(\cdot)=A$ and~$B(\cdot)=B$ are constant). Since~$\UU = \L^\infty([0,T],\R^m)$ and~$\E$ is affine, a control~$u \in \L^\infty([0,T],\R^m)$ is strongly regular if and only if~$\D\E(u)$ is surjective, if and only if~$\E$ is surjective, if and only if the pair~$(A,B)$ satisfies the Kalman condition. This characterization does not depend on $(T,x^0,u)$. Hence, under the Kalman condition, any control~$u \in \L^\infty([0,T],\R^m)$ is strongly regular on any subinterval~$[\t^0,\t^1]\subset[0,T]$ of nonempty interior and from any starting point. Thus, under the Kalman condition and using Remark~\ref{rem14}, if a control~$u \in \L^\infty([0,T],\U)$ takes its values in $\Int(\U)$ along a subinterval~$[\t^0,\t^1]\subset[0,T]$ of nonempty interior, then~$u$ is weakly~$\U$-regular (and even strongly~$\U$-regular if~$\U$ is convex) on~$[0,T]$ from any starting point.
\end{remark}

We now introduce a last notion which will be instrumental in order to relax the convexity assumption made on~$\U$ in our main result (see Remark~\ref{remequalityinequalityU} and~\ref{remequalityinequalityU2} further for details).

\begin{definition}[parameterization of~$\U$]\label{def_paramU}
We say that~$\U$ is \emph{parameterizable} by a nonempty subset~$\U'$ of~$\R^{m'}$, with~$m' \in \N^*$, if there exists a~$\C^1$ mapping~$\varphi : \R^{m'} \to \R^m$ satisfying~$\varphi (\U') = \U$ and, for every~$u \in \L^\infty([0,T],\U)$, there exists~$u' \in  \L^\infty([0,T],\U')$ such that~$u = \varphi \circ u'$.
\end{definition}

\begin{example}
Using a standard measurable selection theorem, we see that the two-dimensional unit circle~$\U = \{ (u_1,u_2) \in \R^2 \mid u_1^2 + u_2^2 = 1 \}$ is parameterizable by the interval~$[0,2\pi]$.
\end{example}

In the context of Definition~\ref{def_paramU}, the control system~\eqref{contsys} has the same trajectories as the control system~\eqref{contsysprime} given by
\begin{equation}\label{contsysprime}\tag{CS'}
\dot{x}'(t) = f' ( x'(t), u'(t),t ), \quad \text{a.e.\ } t \in [0,T],
\end{equation}
starting at the same initial point~$x^0$, where the dynamics~$f' : \R^n \times \R^{m'} \times [0,T] \to \R^n$ is defined by~$f'(x',u',t) = f(x',\varphi(u'),t)$ for all~$(x',u',t) \in \R^n \times \R^{m'} \times [0,T]$ and where~$\U'$ is the control constraint set. Precisely, for a control~$u \in \UU \cap \L^\infty([0,T],\U)$, any control~$u' \in \L^\infty([0,T],\U')$ satisfying~$u = \varphi \circ u'$ belongs to the set~$\UU'$ of all admissible controls for~\eqref{contsysprime}, and~$x'_{u'} = x_u$. Furthermore, by the Hamiltonian characterization, if~$u$ is weakly~$\U$-regular for~\eqref{contsys}, then~$u'$ is weakly~$\U'$-regular for~\eqref{contsysprime}. We say that weak~$\U$-regularity is \textit{preserved} by parameterization. However, when~$\U$ and~$\U'$ are convex, strong~$\U$-regularity may not be preserved by parameterization, as shown in the following example.

\begin{example}
Consider the framework of Example~\ref{example2}. By the Hamiltonian characterization, the constant control~$u \equiv 0$ is strongly~$\U$-regular. Considering the parameterization of~$\U$ by itself, with the $\C^1$ mapping~$\varphi : \R \to \R$ defined by~$\varphi (u') = u'^3$ for every~$u' \in \R$, we recover the control system considered in Example~\ref{example3} in which the constant control~$u' \equiv 0$, which satisfies~$u = \varphi \circ u'$, is weakly~$\U'$-singular.
\end{example}

\section{Robustness under control sampling of reachability in fixed time}\label{sec_comments}

When dealing with controls~$u \in \L^\infty([0,T],\U)$, the control system~\eqref{contsys} is said to be with \textit{permanent controls}, in the sense that the control value can be modified at any real time~$t \in [0,T]$. Otherwise, when dealing with piecewise constant controls~$u \in \PC^\T([0,T],\U)$, for a given partition~$\T = \{ t_i \}_{i=0,\ldots,N}$ of~$[0,T]$, the control system~\eqref{contsys} is said to be with \textit{sampled-data controls} (see~\cite{Nesic2001}) which are a particular case of \textit{nonpermanent controls}, in the sense that the control value can be modified only at the \textit{sampling times}~$t_i \in \T$ and remains frozen on each \textit{sampling interval}~$[t_i,t_{i+1})$. 

In~\cite{BT2017} we proved that the optimal sampled-data control of a general unconstrained linear-quadratic problem converges pointwisely to the optimal permanent control when the norm of the corresponding partition converges to zero. In an ongoing work we extend this result to a general nonlinear setting, moreover under convex control constraints and with fixed endpoint. For this purpose, robustness under control sampling of reachability in fixed time of the fixed endpoint has to be investigated. This issue has motivated the present work.


For any control~$u \in \UU \cap \L^\infty ([0,T],\U)$, we introduce the properties~\eqref{eqPu} and~\eqref{eqQu} defined by
\begin{multline}\label{eqPu}\tag{$ \P_u $}
\exists \delta > 0, \quad \forall \T \in \PP, \; \Vert \T \Vert \leq \delta, \\ x_u(T) \in \E ( \UU \cap \PC^\T ([0,T],\U))
\end{multline}
and
\begin{equation}\label{eqQu}\tag{$ \P'_u $}
\exists \T \in \PP, \quad x_u(T) \in \E ( \UU \cap \PC^\T ([0,T],\U)),
\end{equation}
where~$\PP$ is the set of all partitions of~$[0,T]$ and where~$\E ( \UU \cap \PC^\T ([0,T],\U))$ is the~\textit{$\PC^\T_\U$-accessible set}. Example~\ref{example1} shows that, for a given control~$u \in \UU \cap \L^\infty ([0,T],\U)$, Property~\eqref{eqQu} is not satisfied in general. Most of the literature focuses on establishing sufficient conditions for properties related to Property~\eqref{eqQu} (see Remark~\ref{remgrasse} further for details and references). One of the novelties of the present work is to provide sufficient conditions for the stronger Property~\eqref{eqPu}. The interest of the threshold~$\delta > 0$ in Property~\eqref{eqPu} (which is not considered in Property~\eqref{eqQu}) is twofold. On one hand, its existence is instrumental to extend the convergence result obtained in~\cite{BT2017} to a general nonlinear setting under convex control constraints and with fixed endpoint, precisely in order to guarantee that the corresponding optimal sampled-data control problem is feasible for partitions of sufficiently small norm. On the other hand, the nonexistence of such a threshold~$\delta > 0$ implies that~$\PC^\T_\U$-reachability in time~$T$ from~$x^0$ of the final point~$x_u(T) $ is sensitive to small perturbations of the partition~$\T$ of~$[0,T]$, in the sense of the next proposition (proved in Section~\ref{secproof_propsensitive} and illustrated in Remark~\ref{remillsensitive} further).

\begin{proposition}\label{propsensitive}
Let $u \in \UU \cap \L^\infty ([0,T],\U)$. If Property~\eqref{eqPu} is not satisfied, then, for any partition~$\T=\{ t_i \}_{i=0,\ldots,N}$ of~$[0,T]$ and any~$\eps > 0$, there exists a partition~$\T^\eps = \{ t^\eps_i \}_{i=0,\ldots,N}$ of~$[0,T]$ such that~$\vert t^\eps_i - t_i \vert < \eps$ for all~$i \in \{ 1,\ldots,N-1 \}$ and such that~$x_u(T)$ is not~$\PC^{\T^\eps}_\U$-reachable in time~$T$ from~$x^0$.
\end{proposition}

This section is organized as follows. In Section~\ref{secfirstinvestigations} we first investigate the condition that~$x_u(T)$, for some~$u \in \UU \cap \L^\infty ([0,T],\U)$, belongs to the interior of the~$\L^\infty_\U$-accessible set. Our main result (Theorem~\ref{thmmain}), which is valid under the stronger condition that~$u$ is weakly~$\U$-regular, is discussed in Section~\ref{secmainresultandcomments}.

\subsection{Final point in the interior of the~$\L^\infty_\U$-accessible set}\label{secfirstinvestigations}

Let~$u \in \UU \cap \L^\infty ([0,T],\U)$. Here we focus on the condition that~$x_u(T)$ belongs to the interior of the~$\L^\infty_\U$-accessible set. The next example, based on a commensurability rigidity, shows that it is not a sufficient condition for Property~\eqref{eqPu}.

\begin{example}\label{example4}
Take~$T=4$, $n=m=1$, $\U = \{ 0 , 1 \}$ and~$f(x,u,t)=u$ for all~$(x,u,t) \in \R \times \R \times [0,T]$. The target point~$x^1 = \pi$ is $\L^\infty_\U$-reachable in time~$T$ from the starting point~$x^0 = 0$ with the control~$u(t)=1$ for a.e.\ $t \in [0,\pi]$ and~$u(t)=0$ for a.e.\ $t \in [\pi,4]$. By the Hamiltonian characterization, the control~$u$ is weakly~$\U$-regular and thus~$x^1 =x_u(T)$ belongs to the interior of the~$\L^\infty_\U$-accessible set (see Proposition~\ref{prop_wUreginterior}). However, for any given partition~$\T$ of~$[0,T]$, $x^1 $ belongs to the~$\PC^\T_\U$-accessible set if and only if there exists a subfamily of sampling intervals associated with~$\T$ whose sum of lengths is equal to~$\pi$. As a consequence, for any partition~$\T$ of~$[0,T]$ containing only rational sampling times (with norm~$\Vert \T \Vert$ arbitrarily small),~$x^1 $ is not~$\PC^\T_\U$-reachable in time~$T$ from~$x^0$. We conclude that Property~\eqref{eqPu} is not satisfied (while Property~\eqref{eqQu} is).
\end{example}

In Example~\ref{example4}, the set~$\U$ is not convex. However note that another counterexample, in which~$\U$ is convex, is provided in Appendix~\ref{appexamplegrasse}.

\begin{remark}\label{remillsensitive}
Example~\ref{example4} illustrates Proposition~\ref{propsensitive} in the sense that, given any partition~$\T=\{ t_i \}_{i=0,\ldots,N}$ of~$[0,T]$ (even such that the target point~$x^1$ is $\PC^\T_\U$-reachable in time~$T$ from~$x^0$) and given any~$\eps > 0$, there always exists a partition~$\T^\eps =\{ t ^\eps_i \}_{i=0,\ldots,N}$ of~$[0,T]$ containing only rational sampling times such that~$\vert t^\eps_i - t_i \vert < \eps$ for all~$i \in \{ 1,\ldots,N-1 \}$, and thus such that~$x^1$ is not~$\PC^{\T^\eps}_\U$-reachable in time~$T$ from~$x^0$. We provide in the following example a similar illustration of Proposition~\ref{propsensitive} with~$\U$ convex.
\end{remark}

\begin{example}\label{example6}
Take~$T=4$, $n=2$, $m=1$, $\U = [0,1]$ and~$f((x_1,x_2),u,t)=(u,u^2)$ for all~$((x_1,x_2),u,t) \in \R^2 \times \R \times [0,T]$. Consider the starting point~$x^0 = 0_{\R^2}$. The point~$x_u(T)$ belongs to the segment~$\{ (x_1,x_2) \in \R^2 \mid 0 \leq x_1 = x_2 \leq 4 \}$ if and only if the corresponding control~$u \in \UU \cap \L^\infty([0,T],\U)$ takes its values in~$\{ 0 , 1 \}$. As a consequence, by considering the target point~$x^1=(\pi,\pi)$, we find the same conclusions as in Example~\ref{example4}.
\end{example}

In the one-dimensional case~$n=1$, the next proposition is obtained (see the proof in Section~\ref{secproof_propn1} based on the fact that one-dimensional connected sets are convex).

\begin{proposition}\label{propn1}
Assume that $n=1$, that~$\U$ is convex\footnote{Actually assuming that~$\U$ is connected is sufficient.} and that~$\UU = \L^\infty([0,T],\R^m)$. If~$x_u(T)$, for some~$u \in \L^\infty ([0,T],\U)$, belongs to the interior of the~$\L^\infty_\U$-accessible set, then Property~\eqref{eqPu} is satisfied.
\end{proposition}

\subsection{Comments on Theorem~\ref{thmmain}}\label{secmainresultandcomments}
Let~$u \in \UU \cap \L^\infty ([0,T],\U)$. This section focuses on the condition that~$u$ is weakly~$\U$-regular. Example~\ref{example4} shows that it is not a sufficient for Property~\eqref{eqPu} in general. However note that our main result (Theorem~\ref{thmmain}) states that, when~$\U$ is convex, it is a sufficient condition for Property~\eqref{eqPu}. 

\begin{example}\label{example5}
Take~$T=18$, $n=2$, $m=1$, $\U = [-1,1]$ and~$f((x_1,x_2),u,t) = (x_2,u)$ for all~$((x_1,x_2),u,t) \in \R^2 \times \R \times [0,T]$. The target point~$x^1 = (0,0)$ is~$\L^\infty_\U$-reachable in time~$T$ from the starting point~$x^0 = (78,0)$ with the control~$u(t)=-1$ for a.e.\ $t \in [0,6]$, $u(t) = \frac{t-9}{3}$ for a.e.\ $t \in [6,12]$ and~$u(t)=1$ for a.e.\ $t \in [12,18]$. The control~$u$ is weakly~$\U$-regular by Remark~\ref{rem97}. Therefore, by Theorem~\ref{thmmain}, there exists~$\delta > 0$ such that~$x^1$ is~$\PC^\T_\U$-reachable in time~$T$ from~$x^0$ for any partition~$\T$ of~$[0,T]$ satisfying~$\Vert \T \Vert \leq \delta$.
\end{example}

In this paper we provide two different proofs of Theorem~\ref{thmmain}. A first proof is done in Section~\ref{sec_firstproof}, under the stronger condition that $u$ is strongly~$\U$-regular. This proof uses results of Section~\ref{sec_regwith} (in particular, conic~$\L^\infty$-perturbations of $u$) and, as explained in Remark~\ref{remfirstproof} further, we resort to truncated dynamics. In the second proof, given in Section~\ref{sec_secondproof}, we treat the case where~$u$ is assumed to be (only) weakly~$\U$-regular. This proof uses results of Section~\ref{sec_wregwith} (in particular, needle-like variations of $u$) and, as explained in Remark~\ref{remsecondproof}, we resort to the Brouwer fixed-point theorem. We think the two proofs are interesting, not only for pedagogical reasons but also because the different techniques that we introduce may be useful for other issues.
Note that both proofs use, at some step, the conic implicit function theorem~\cite[Theorem~1]{antoine1990} and averaging operators which project any integrable function onto a piecewise constant function. 

\begin{remark}\label{remfirstproof}
The first proof of Theorem~\ref{thmmain}, given in Section~\ref{sec_firstproof} under the strong~$\U$-regularity assumption, relies on the conic implicit function theorem~\cite[Theorem~1]{antoine1990}. However this theorem must be used in the Banach space~$\L^s([0,T],\R^m)$, for some~$1 < s < +\infty$, and not in~$\L^\infty([0,T],\R^m)$. This is because it is not true that any function in~$\L^\infty([0,T],\R^m)$ can be approximated in~$\L^\infty$-norm by piecewise constant functions, while it can be in~$\L^s$-norm with any~$1 \leq s < +\infty$ (see Appendix~\ref{app2}). This leads us to extend the end-point mapping to~$\L^s([0,T],\R^m)$ which makes no sense a priori because the control system~\eqref{contsys} is nonlinear.\footnote{For example, take~$n=m=1$, $s=2$ and~$f(x,u,t)=u^4$ for all~$(x,u,t) \in \R \times \R \times [0,T]$. Then considering~$\L^2$-controls makes no sense.} To overcome this difficulty, we introduce in Appendix~\ref{apptruncated} a truncated version of the dynamics~$f$, vanishing outside of a sufficiently large compact subset of~$\R^n \times \R^m$. Then the corresponding truncated end-point mapping is well defined on~$\L^s([0,T],\R^m)$, but is not Fr\'echet-differentiable when~$s=1$. However it is of class~$\C^1$ when~$1 < s < +\infty$ and the surjectivity of the differential of the truncated end-point mapping in~$\L^s$-norm can be related to the surjectivity of the differential in~$\L^\infty$-norm of the nontruncated end-point mapping. This is a key technical point in the first proof of Theorem~\ref{thmmain}.
\end{remark}

\begin{remark}\label{remsecondproof}
The second proof of Theorem~\ref{thmmain}, given in Section~\ref{sec_secondproof} under the weak~$\U$-regularity assumption, relies on the conic implicit function theorem~\cite[Theorem~1]{antoine1990} applied to the end-point mapping restricted to a multiple needle-like variation (as in the proof of Proposition~\ref{prop_wUreginterior}). This second proof of Theorem~\ref{thmmain} also uses the Brouwer fixed-point theorem. Like in~\cite[Lemma~3.1]{Grasse1981} or in~\cite[Lemma~7]{AgrachevCaponigro2010}, the main idea is that, under appropriate assumptions, local surjectivity of a continuous mapping is preserved under small perturbations. In our context, local surjectivity of the above restriction of the end-point mapping is preserved under the perturbation due to the composition with an averaging operator (see Appendix~\ref{app3}) which project any control with values in~$\U$ onto a piecewise constant control with values in~$\U$.
\end{remark}

\begin{remark}\label{remgrasse}
Theorem~\ref{thmmain} establishes robustness under control sampling of reachability \emph{in fixed time}. If one does not fix the final time, robustness under control sampling of reachability is already known, and this remark is dedicated to the remarkable series of papers~\cite{grasse1992,grasse1995,grasse1990,sussmann1987} by Grasse and Sussmann (see also references therein) on reachability and controllability with piecewise constant controls. 
\begin{enumerate}[label=\rm{(\roman*)}]
\item It is established in~\cite[Theorem~4.2]{sussmann1987} that normal reachability of a target point in the state space implies normal reachability with a piecewise constant control. Roughly speaking, normal reachability is reachability under a surjectivity assumption which is similar to the notion of regularity considered in the present work.
\item With another point of view (not based on a surjectivity property), it is established in~\cite[Theorem~3.17]{grasse1990} that, under global controllability, the controllability can be achieved with piecewise constant controls.
\item In~\cite[Remark~3.5]{grasse1992} it is noted that if a point of the state space is normally reachable in time less than~$T$, then it belongs to the interior of the reachable set with piecewise-constant controls in time less than~$T$. 
\item It is proved in~\cite[Corollary~4.4]{grasse1995} that, if the initial condition belongs to the interior of the reachable set, then this reachable set coincides with the reachable set with piecewise constant controls. 
\end{enumerate}
Our main result (Theorem~\ref{thmmain}) differs from the above results for two reasons.  First, as underlined above, the final time~$T $ is fixed in our work, while it is not in the abovementioned references. For instance, in ~\cite[Theorem~4.2]{sussmann1987}, normal reachability with a piecewise constant control is established, but a priori for a different final time~$T' $ (and indeed, inspection of the proof shows that, in general,~$T'\neq T$). Second, our main result (Theorem~\ref{thmmain}) states the existence of a threshold~$\delta > 0$ for which reachability (exactly at time~$T$) of a target point with a piecewise constant control is guaranteed for any partition~$\T$ satisfying~$\Vert \T \Vert \leq \delta$. The existence of this threshold (which is not considered in the abovementioned works) is of particular interest when considering refinements of partitions (for convergence results for instance) and for robustness of reachability under small perturbations of the partition (see Proposition~\ref{propsensitive}). Furthermore, since the inclusion~$\subset$ is not a total order over~$\PP$, it may occur that~$\Vert \T_2 \Vert \leq \Vert \T_1 \Vert$ while~$\T_1 \not\subset \T_2$. In the above references, it is not guaranteed that reachability of a target point with a~$\T_1$-piecewise constant control implies reachability with a~$\T_2$-piecewise constant control. With the conclusion of Theorem~\ref{thmmain}, when~$\Vert \T_1 \Vert \leq \delta$, it is guaranteed.
\end{remark}

\begin{remark}\label{remequalityinequalityU}
The convexity assumption made on~$\U$ in Theorem~\ref{thmmain} can be relaxed. Indeed let us prove that Theorem~\ref{thmmain} is still true when~$\U$ is assumed to be (only) \textit{convex by parameterization}, i.e., when~$\U$ is parameterizable (see Definition~\ref{def_paramU}) by a nonempty convex subset~$\U'$ of~$\R^{m'}$ for some~$m' \in \N^*$ (see examples in Remark~\ref{remequalityinequalityU2}). In that context, for a control~$u \in \UU \cap \L^\infty([0,T],\U)$ that is weakly~$\U$-regular, there exists~$u' \in \UU' \cap \L^\infty([0,T],\U')$ such that~$u = \varphi \circ u'$ and~$u'$ is weakly~$\U'$-regular for the control system~\eqref{contsysprime}. Since~$\U'$ is convex, there exists by Theorem~\ref{thmmain} a threshold~$\delta > 0$ such that~$x'_{u'}(T) = x'_{v'_\T}(T)$, for some~$v'_\T \in \UU' \cap \PC^\T([0,T],\U')$, for all partitions~$\T$ satisfying~$\Vert \T \Vert \leq \delta$. Introducing~$v_\T = \varphi \circ v'_\T \in \UU \cap  \PC^\T([0,T],\U)$, we obtain that~$x_u(T) = x'_{u'}(T) = x'_{v'_\T}(T) = x_{v_\T}(T)$ for all partitions~$\T$ satisfying~$\Vert \T \Vert \leq \delta$.
\end{remark}

\begin{remark}\label{remequalityinequalityU2}
If~$\U$ is convex by parameterization (see Remark~\ref{remequalityinequalityU}), then~$\U$ must be connected. Actually a quite large class of connected sets are convex by parameterization. For example, in the two-dimensioncal case~$m=2$, the unit circle~$\U = \{ (u_1,u_2) \in \R^2 \mid u_1^2 + u_2^2 = 1 \}$, the donut-shaped set~$\U = \{ (u_1,u_2) \in \R^2 \mid 1 \leq u_1^2 + u_2^2 \leq 4 \}$ or the cross-shaped set~$\U = ( [-1,1] \times \{ 0 \} ) \cup ( \{ 0 \} \times [-1,1])$ are nonconvex connected sets that are convex by parameterization. For these sets, the conclusion of Theorem~\ref{thmmain} holds true. However, adapting Example~\ref{example4}, note that the conclusion of Theorem~\ref{thmmain} fails in general if~$\U$ is \textit{strongly nonconnected}, i.e., when it can be written as~$\U = \U_1 \cup \U_2$, where~$\U_1$ and~$\U_2$ are nonempty, and there exists a~$\C^1$ mapping~$\Theta : \R^m \to \R$ taking the value~$0$ on~$\U_1$ and the value~$1$ on~$\U_2$.\footnote{For example, when~$m=1$, the set~$(-\infty,0] \cup [1,+\infty)$ is strongly nonconnected, while the set~$(-\infty,0) \cup (0,+\infty)$ is nonconnected (but not strongly).} An open question is to extend Theorem~\ref{thmmain} to sets~$\U$ that are neither convex by parameterization, nor strongly nonconnected. We emphasize that our proof of Theorem~\ref{thmmain}, when~$\U$ is convex, uses the averaging operators introduced in Appendix~\ref{app3}, which project any control with values in~$\U$ onto a piecewise constant control with values in~$\U$ (see Proposition~\ref{propaverage3}). When $\U$ is not convex, one has to consider other operators: one way may be to follow the approach based on the Lusin theorem~\cite{lusin1912} as developed in Appendix~\ref{app2}.
\end{remark}

\begin{remark}\label{rem_relaxregularityonf}
Several statements in the present paper do not require that the dynamics~$f$ is of class~$\C^1$ with respect to~$u$. Actually this assumption is required (only) when~$\nabla_u f$ has to be considered (such as in Sections~\ref{sec_regwithout} and~\ref{sec_regwith} where we use conic~$\L^\infty$-perturbations). When using needle-like variations (which are~$\L^1$-perturbations) such as in Section~\ref{sec_wregwith}, it is only required that~$f$ is of class~$\C^1$ with respect to~$x$ and is Lipschitz continuous with respect to~$(x,u)$ on any compact subset of~$\R^n \times \R^m \times [0,T]$. In particular the conclusion of Theorem~\ref{thmmain} remains true in that context.\footnote{By Remark~\ref{rem_relaxregularityonf}, Definition~\ref{def_paramU} (resp., the notion of strongly nonconnected set introduced in Remark~\ref{remequalityinequalityU2}) can be relaxed by considering a mapping~$\varphi$ (resp., $\Theta$) that is (only) Lipschitz continuous on any compact subset of~$\R^{m'}$ (resp., of~$\R^m$).}
\end{remark}

\begin{remark}\label{rem_conenonimage}
As far as we know, the $\U$-Pontryagin cone of a control~$u \in \UU \cap \L^\infty([0,T],\U)$ cannot be written as the range of a differential~$\D \E (u)$ taken in an appropriate sense. Indeed, we explain in Section~\ref{secproof_wUreginterior} how~$\Pont_\U[u]$ can be generated using multiple needle-like variations which are~$\L^s$-perturbations for any~$1 \leq s < +\infty$. Nevertheless, even using truncated dynamics in order to work in~$\L^s([0,T],\R^m)$ for some~$1 \leq s < +\infty$, we explain in Appendix~\ref{apptruncated} that the truncated end-point mapping is not Fr\'echet-differentiable when~$s=1$ and, when~$1 < s < +\infty$, the Fr\'echet differential of the truncated end-point mapping generates (only) weak $\U$-variation vectors. We conclude this comment by referring to the work of Gamkrelidze in~\cite{Gamkrelidze1978} in which classical controls are embedded in the set of Radon measures. With this nonstandard approach, it is proved that~$\Pont_\U[u]$ is contained in the range of the differential of the end-point mapping considered on the set of Radon measures. Unfortunately the above embedding has a convexification effect on the dynamics~$f$ and, as a result, the inclusion is (only) strict in general.
\end{remark}

\section{Proofs of results of Section~\ref{sec_recap}}\label{secproofs1}

This section is dedicated to proving the results of Section~\ref{sec_recap}. Most of the following proofs are known in the literature. They are recalled here because the techniques and results developed hereafter will be helpful at several occasions in Section~\ref{secproofs2} (devoted to proving the new results presented in Section~\ref{sec_comments}).

In what follows, when~$(\ZZ,\d_\ZZ)$ is a metric set, we denote by~$\B_{\ZZ}(z,\rho)$ (resp.\ $\BB_{\ZZ}(z,\rho)$) the  open ball (resp.\ closed ball) centered at some~$z \in \ZZ$ of some radius~$\rho \geq 0$.

\subsection{Proof of Proposition~\ref{prop_reginterior}}\label{secproof_reginterior}

Let~$u \in \UU$ be strongly regular. By Definition~\ref{def_reg} there exists a~$n$-tuple~$\bbar{v} = \{ v_j  \}_{j=1,\ldots,n}$ of elements of~$\L^\infty([0,T],\R^m)$ such that $\D\E(u)  \cdot v_j = e_j$ for all $j \in \{ 1,\ldots,n \}$, where~$\{ e_j \}_{j=1,\ldots,n}$ is the canonical basis of~$\R^n$. We define the mapping~$\Phi : \R^n \times [-\beta,\beta]^{n} \longrightarrow \R^n $ by
$$ 
\Phi(z,\bbar{\alpha}) =  \E \bigg( u +\di \sum_{j=1}^{n} \alpha_j v_j \bigg) - z
$$ 
for all~$(z,\bbar{\alpha}) \in \R^n \times [-\beta,\beta]^{n}$, where~$\beta > 0$ is small enough to guarantee that~$u + \sum_{j=1}^{n} \alpha_j v_j \in \UU$ for all~$\bbar{\alpha} \in [-\beta,\beta]^{n}$, which is possible because~$\UU$ is an open subset of~$\L^\infty([0,T],\R^m)$. The mapping~$\Phi$ is of class~$\C^1$ and satisfies~$\Phi(x_u(T),0_{\R^{n}}) = 0_{\R^n}$ and~$\frac{\partial \Phi}{\partial \bbar{\alpha}} (x_u(T),0_{\R^{n}})  = \mathrm{Id}_{\R^n}$ which is invertible. By the implicit function theorem, there exists an open neighborhood~$\VV$ of~$x_u(T)$ and a~$\C^1$ mapping~$\bbar{\alpha} : \VV \to [-\beta,\beta]^{n}$ satisfying~$\bbar{\alpha} ( x_u(T) ) = 0_{\R^n}$ and~$\Phi(z , \bbar{\alpha} (z) ) = 0_{\R^n}$ for all~$z \in \VV$. Then it suffices to introduce the~$\C^1$ mapping~$V : \VV \to \UU$ defined by~$V(z) = u + \sum_{j=1}^{n} \alpha_j (z) v_j$ for all~$z \in \VV$.

\subsection{Proof of Proposition~\ref{prop_charactsing}}\label{secproof_charactsing}


\begin{lemma}\label{lemcharactweakextremal}
Let $u \in \UU$ and~$p \in \AC ([0,T],\R^n)$ be a solution to~\eqref{eqAE}. The following statements are equivalent:
\begin{enumerate}[label=\rm{(\roman*)}]
\item\label{lemcharactweakextremal_i} $(x_u,u,p)$ is a weak extremal lift of the pair~$(x_u,u)$;
\item\label{lemcharactweakextremal_iii} $\langle p(T) , \D\E(u) \cdot v \rangle_{\R^n} = 0$ for all~$v \in \L^\infty([0,T],\R^m)$.
\end{enumerate}
\end{lemma}

\begin{proof}
We set~$h_v(t) = \langle p(t) , w^u_{v}(t) \rangle_{\R^n}$ for all~$t \in [0,T]$ and all~$v \in \L^\infty([0,T],\R^m)$, where~$w^u_{v}$ is defined after~\eqref{eq1}. Therefore~\ref{lemcharactweakextremal_iii} is equivalent to~$h_v(T) = 0$ for all~$ v \in \L^\infty([0,T],\R^m)$. For all~$v \in \L^\infty([0,T],\R^m)$, note that~$h_v(0)=0$ and, using the adjoint equation~\eqref{eqAE}, that
$$ 
\dot{h}_v(t) = \langle \nabla_u H(x_u(t),u(t),p(t),t) , v(t) \rangle_{\R^m}  
$$ 
for a.e.\ $t \in [0,T]$. Now let us to prove that~\ref{lemcharactweakextremal_i} is equivalent to~\ref{lemcharactweakextremal_iii}. First let us assume~\ref{lemcharactweakextremal_i}. From the null Hamiltonian gradient condition~\eqref{eqNHG}, we have~$\dot{h}_v (t) = 0$ for a.e.\ $t \in [0,T]$ and thus~$h_v(T) = h_v(0)= 0$ for all~$v \in \L^\infty([0,T],\R^m)$, which gives~\ref{lemcharactweakextremal_iii}. Now, assuming~\ref{lemcharactweakextremal_iii}, we have~
$
\int_0^T \langle \nabla_u H (x_u(t),u(t),p(t),t) , v(t) \rangle_{\R^m} dt = h_v(T) =  0
$
for every~$v \in \L^\infty([0,T],\R^m)$. We deduce the null Hamiltonian gradient condition~\eqref{eqNHG}, which gives~\ref{lemcharactweakextremal_i}.
\end{proof}

Let us prove Proposition~\ref{prop_charactsing}. Let~$u \in \UU$. First, assume that~$u$ is~weakly singular, i.e., $\mathrm{Ran}(\D\E(u))$ is a proper subspace of~$\R^n$. Hence there exists~$\psi \in \R^n \backslash \{ 0_{\R^n} \}$ such that~$\langle \psi , \D \E(u) \cdot v \rangle_{\R^n} = 0$ for all~$ v \in \L^\infty([0,T],\R^m)$. Considering~$p \in \AC([0,T],\R^n)$ the unique solution to~\eqref{eqAE} ending at~$p(T) = \psi$ (in particular~$p$ is not trivial), we obtain that~$\langle p(T) , \D \E(u) \cdot v \rangle_{\R^n} = 0$ for all~$ v \in \L^\infty([0,T],\R^m)$. By Lemma~\ref{lemcharactweakextremal}, $(x_u,u,p)$ is a nontrivial weak extremal lift of~$(x_u,u)$. Conversely, assume that~$u$ is strongly regular, i.e., $\mathrm{Ran}(\D\E(u))= \R^n$. By contradiction let us assume that~$(x_u,u)$ admits a nontrivial weak extremal lift~$(x_u,u,p)$. Then there exists~$v \in \L^\infty([0,T],\R^m)$ such that~$ \D\E(u) \cdot v = p(T)$. It follows from Lemma~\ref{lemcharactweakextremal} that~$ \Vert p(T) \Vert_{\R^n}^2 = 0$ and thus~$p(T) = 0_{\R^n}$. Since the adjoint equation~\eqref{eqAE} is linear, it follows that~$p$ is trivial, which raises a contradiction.

\subsection{Proof of Proposition~\ref{prop_Ureginterior}}\label{secproof_Ureginterior}

\begin{lemma}\label{lemconetangent}
Assume that~$\U$ is convex and let~$u \in \L^\infty([0,T],\U)$. We have
\begin{multline*}
\TT_{\L^\infty_\U}[u] = \big\{ v \in \L^\infty([0,T],\R^m)  \mid \\
\exists \beta > 0, \; u + \beta v \in \L^\infty([0,T],\U) \big\} . 
\end{multline*}
Furthermore, for every~$J \in \N^*$, we have
$$ 
u + \sum_{j=1}^J \alpha_j v_j \in \L^\infty([0,T],\U) 
$$ 
for every~$\alpha_j \in [0,\frac{\beta_j}{J}]$, where~$v_j \in \TT_{\L^\infty_\U}[u]$ and~$\beta_j > 0$ is such that~$u + \beta_j v_j \in \L^\infty([0,T],\U)$ for every~$j \in \{ 1,\ldots,J \}$.
\end{lemma}

Lemma~\ref{lemconetangent} is obvious. Assume that~$\U$ is convex and let us prove Proposition~\ref{prop_Ureginterior}. Let~$u \in \UU \cap \L^\infty([0,T],\U)$ be strongly~$\U$-regular. By Definition~\ref{def_Ureg}, there exists a $2n$-tuple~$\bbar{v} = \{ v_j  \}_{j=1,\ldots,2n}$ of elements of~$\TT_{\L^\infty_\U}[u]$ such that 
\begin{equation}\label{condtfic0}
\D\E(u)  \cdot v_j = e_j  \quad \text{and} \quad \D\E(u)  \cdot v_{n+j} = -e_j 
\end{equation}
for every $j \in \{ 1,\ldots,n \}$, where~$\{ e_j \}_{j=1,\ldots,n}$ is the canonical basis of~$\R^n$. We define the map
$\Phi : \R^n \times [0,\beta]^{2n} \longrightarrow \R^n $
by
$$
\Phi(z,\bbar{\alpha}) =  \E \bigg( u + \di \sum_{j=1}^{2n} \alpha_j v_j \bigg) - z 
$$
for all~$(z,\bbar{\alpha}) \in \R^n \times [0,\beta]^{2n}$, where~$\beta > 0$ is small enough to guarantee that~$u + \sum_{j=1}^{2n} \alpha_j v_j \in \UU \cap \L^\infty([0,T],\U)$ for every~$\bbar{\alpha} \in [0,\beta]^{2n}$, which is possible by Lemma~\ref{lemconetangent} and because~$\UU$ is an open subset of~$\L^\infty([0,T],\R^m)$. The mapping~$\Phi$ is of class~$\C^1$ and satisfies~$\Phi(x_u(T),0_{\R^{2n}}) = 0_{\R^n}$ and~$\frac{\partial \Phi}{\partial \bbar{\alpha}} (x_u(T),0_{\R^{2n}}) \cdot \R^{2n}_+  = \R^n$ thanks to~\eqref{condtfic0}. From the conic implicit function theorem~\cite[Theorem~1]{antoine1990}, there exists an open neighborhood~$\VV$ of~$x_u(T)$ and a continuous mapping~$\bbar{\alpha} : \VV \to [0,\beta]^{2n}$ satisfying~$\bbar{\alpha} ( x_u(T) ) = 0_{\R^{2n}}$ and~$\Phi(z , \bbar{\alpha} (z) ) = 0_{\R^n}$ for all~$z \in \VV$. Then it suffices to introduce the continuous mapping~$V : \VV \to \UU \cap \L^\infty([0,T],\U)$ defined by~$V(z) = u + \sum_{j=1}^{2n} \alpha_j (z) v_j$ for all~$z \in \VV$.

\subsection{Proof of Proposition~\ref{prop_charactUsing}}\label{secproof_charactUsing}


\begin{lemma}\label{lemcharactweakUextremal}
Assume that~$\U$ is convex. Let $u \in \UU \cap \L^\infty([0,T],\U)$ and~$p \in \AC ([0,T],\R^n)$ be a solution to~\eqref{eqAE}. The following statements are equivalent:
\begin{enumerate}[label=\rm{(\roman*)}]
\item\label{lemcharactweakUextremal_i} $(x_u,u,p)$ is a weak $\U$-extremal lift of the pair~$(x_u,u)$;
\item\label{lemcharactweakUextremal_iii} $\langle p(T) , \D\E(u) \cdot (v-u) \rangle_{\R^n} \leq 0$ for all~$ v \in \L^\infty([0,T],\U)$;
\item\label{lemcharactweakUextremal_iv} $\langle p(T) , \D\E(u) \cdot v \rangle_{\R^n} \leq 0$ for all~$v \in \TT_{\L^\infty_\U}[u]$.
\end{enumerate}
\end{lemma}

\begin{proof}
The equivalence between~\ref{lemcharactweakUextremal_iii} and~\ref{lemcharactweakUextremal_iv} follows from the definition of~$\TT_{\L^\infty_\U}[u]$ (see Definition~\ref{def_Ureg}). Note that~\ref{lemcharactweakUextremal_iii} is equivalent to~$h_{v-u}(T) \leq 0$ for all~$ v \in \L^\infty([0,T],\U)$ (see the definition of~$h_{v-u}$ in the proof of Lemma~\ref{lemcharactweakextremal}). Now let us prove that~\ref{lemcharactweakUextremal_i} is equivalent to~\ref{lemcharactweakUextremal_iii}. First let us assume~\ref{lemcharactweakUextremal_i}. We infer from the Hamiltonian gradient condition~\eqref{eqHG} that~$\dot{h}_{v-u} (t) \leq 0$ for a.e.\ $t \in [0,T]$ and thus~$h_{v-u}(T) \leq h_{v-u}(0)= 0$ for all~$v \in \L^\infty([0,T],\U)$, which gives~\ref{lemcharactweakUextremal_iii}. Now, assuming~\ref{lemcharactweakUextremal_iii}, we have~$\int_0^T \langle \nabla_u H (x_u(t),u(t),p(t),t) , v(t)-u(t) \rangle_{\R^m} dt = h_{v-u} (T) \leq 0$
for every $v \in \L^\infty([0,T],\U)$. Then, for any Lebesgue point~$\t \in [0,T)$ of~$\nabla_u H (x_u,u,p,\cdot) \in \L^\infty([0,T],\R^m)$ and of~$\langle \nabla_u H (x_u,u,p,\cdot),u \rangle_{\R^m} \in \L^\infty([0,T],\R)$ and for any~$\omega \in \U$, taking the needle-like variation~$v=u^\alpha_{(\t,\omega)}  \in \L^\infty([0,T],\U)$ as defined in~\eqref{eq_singleneedle}, we get that~$
\frac{1}{\alpha} \int_\t^{\t+\alpha} \langle \nabla_u H (x_u(t),u(t),p(t),t) , \omega-u(t) \rangle_{\R^m} dt \leq 0
$
for every~$\alpha > 0$ small enough. Taking the limit~$\alpha \to 0^+$, since~$\t$ is an appropriate Lebesgue point, we obtain that~$ \langle \nabla_u H (x_u(\t),u(\t),p(\t),\t) , \omega-u(\t) \rangle_{\R^m}  \leq 0 $. Since~$\t$ and~$\omega$ have been chosen arbitrarily, the Hamiltonian gradient condition~\eqref{eqHG} is satisfied, which gives~\ref{lemcharactweakUextremal_i}.
\end{proof}

Assume that~$\U$ is convex and let us prove Proposition~\ref{prop_charactUsing}. Let~$u \in \UU \cap \L^\infty([0,T],\U)$. Firstly, assume that~$u$ is~weakly $\U$-singular, i.e., $ \D \E(u)(\TT_{\L^\infty_\U}[u])$ is a proper subcone of~$\R^n $. Hence~$0_{\R^n}$ belongs to its boundary and, since~$ \D \E(u) (\TT_{\L^\infty_\U}[u])$ is also convex, by a standard separation argument, there exists~$\psi \in \R^n \backslash \{ 0_{\R^n} \}$ such that~$\langle \psi , \D \E(u) \cdot v \rangle_{\R^n}  \leq 0$ for all~$ v \in \TT_{\L^\infty_\U}[u]$. Considering~$p \in \AC([0,T],\R^n)$ the unique solution to~\eqref{eqAE} ending at~$p(T) = \psi$ (in particular~$p$ is not trivial), we obtain that~$\langle p(T) , \D \E(u) \cdot v \rangle_{\R^n}  \leq 0$ for all~$ v \in \TT_{\L^\infty_\U}[u]$. By Lemma~\ref{lemcharactweakUextremal}, $(x_u,u,p)$ is a nontrivial weak~$\U$-extremal lift of~$(x_u,u)$. Conversely, assume that~$u$ is strongly~$\U$-regular, i.e., $\D \E(u)(\TT_{\L^\infty_\U}[u]) = \R^n$. By contradiction let us assume that~$(x_u,u)$ admits a nontrivial weak~$\U$-extremal lift~$(x_u,u,p)$. There exists~$v \in \TT_{\L^\infty_\U}[u]$ such that~$ \D\E(u) \cdot v = p(T)$. By Lemma~\ref{lemcharactweakUextremal} we get that~$ \Vert p(T) \Vert_{\R^n}^2 \leq 0$ and thus~$p(T) = 0_{\R^n}$. Since the adjoint equation~\eqref{eqAE} is linear, it follows that~$p$ is trivial, which raises a contradiction.

\subsection{Proof of Proposition~\ref{prop_wUreginterior}}\label{secproof_wUreginterior}
Given $u \in \L^\infty([0,T],\R^m)$ and $1 \leq s < +\infty$, we define 
$$ \NNN_{\L^s}(u,\rho,M) = \BB_{\L^s} ( u , \rho ) \cap \BB_{\L^\infty} ( 0_{\L^\infty} , M ) $$
for every~$M \geq \Vert u \Vert_{\L^\infty}$ and every~$\rho > 0$, which corresponds to a usual~$\L^s$-neighborhood of~$u$, truncated with a uniform~$\L^\infty$-bound. 
The following lemmas follow from standard techniques in ordinary differential equations theory. 

\begin{lemma}\label{lemopenL1}
Let~$1 \leq s < +\infty$ and~$u \in \UU$. For any~$M \geq \Vert u \Vert_{\L^\infty}$, there exists~$\rho^M > 0$ such that 
$ \NNN_{\L^s}(u,\rho^M,M)\subset \UU$ and $\Vert x_v - x_u \Vert_\C \leq 1$ for all~$v \in \NNN_{\L^s}(u,\rho^M,M)$ . Moreover the restriction of~$\E $ to~$ \NNN_{\L^s}(u,\rho^M,M)$ is Lipschitz continuous when endowing~$ \NNN_{\L^s}(u,\rho^M,M)$ with the~$\L^s$-metric.
\end{lemma}

\begin{definition}[Multiple needle-like variation]\label{def_multineedle}
Let $u \in \UU \cap \L^\infty([0,T],\U)$. A \emph{package}~$\chi = (\overline{\t},\overline{\omega}) \in \mathcal{L}(f_u)^Q \times \U^R$, with~$Q$, $R \in \N^*$, $Q \leq R$, consists of:
	\begin{itemize}
		\item a $Q$-tuple 
		$ \overline{\t} = \{ \t_q \}_{q = 1, \ldots, Q} \in \mathcal{L}(f_u)^Q $
		such that~$0 \leq \t_1 < \t_2 < \ldots < \t_Q < T$;
		\item a $R$-tuple 
		$ \overline{\omega} = \{ \omega^r_q \}^{ r=1, \ldots, R_q}_{q=1,\ldots,Q} \in \U^R $
		with $R_q \in \N^*$ for all~$q \in \{ 1,\ldots, Q \}$, and~$R = \sum_{q=1}^{Q} R_q$.
	\end{itemize}
The \emph{multiple needle-like variation}~$u^{\overline{\alpha}}_{\chi} \in \L^\infty([0,T],\U)$ of the control~$u$ is defined by
	$$ u^{\overline{\alpha}}_{\chi}(t) = \left\lbrace
	\begin{array}{ll}
	\omega^r_q & \text{along } [ \t_q + \sum_{\ell=1}^{r-1}\alpha^\ell_q ,   \t_q + \sum_{\ell=1}^{r} \alpha^\ell_q ), \\[5pt]
	&  \quad \quad \; \forall r \in \{ 1 , \ldots ,R_q \} , \; \forall q \in \{ 1,\ldots,Q \} , \\[3pt]
	u(t) &  \text{elsewhere,}
	\end{array}
	\right. $$ 
for a.e.\ $t \in [0,T]$ and for all~$\overline{\alpha} \in \R^R_+ $ sufficiently small so that the intervals do not overlap.
\end{definition}

\begin{remark}\label{rembef}
Let~$1 \leq s < +\infty$ and consider the framework of Definition~\ref{def_multineedle}. The mapping~$\bbar{\alpha} \mapsto u^{\overline{\alpha}}_{\chi}$ is continuous when endowing~$\L^\infty([0,T],\U)$ with the~$\L^s$-metric. Taking~$M = \Vert u \Vert_{\L^\infty} + \Vert \bbar{\omega} \Vert_{(\R^m)^R}$ and considering~$\rho^M > 0$ given in Lemma~\ref{lemopenL1}, there exists~$\beta > 0 $ sufficiently small so that~$ u^{\overline{\alpha}}_{\chi} \in \NNN_{\L^s}(u,\rho^M,M) \subset \UU $ for all~$\overline{\alpha} \in [0,\beta]^R$.
\end{remark}

\begin{lemma}\label{lem_multineedle}
In the frameworks of Definition~\ref{def_multineedle} and of Remark~\ref{rembef}, the mapping~$\Psi : [0,\beta]^R \to \R^n$, defined by~$\Psi(\overline{\alpha}) = \E ( u^{\overline{\alpha}}_{\chi} )$ for all~$\overline{\alpha} \in [0,\beta]^R$, satisfies~$\Psi ( 0_{\R^R} ) = x_u(T)$ and is of class~$\C^1$ with 
$$ 
\dfrac{\partial \Psi}{\partial \alpha^r_q} (0_{\R^R}) = w^u_{  \t_q , \omega^r_q } (T) 
$$	
for every~$r \in \{ 1,\ldots,R_q \}$ and every~$q \in \{ 1,\ldots,Q \}$.
\end{lemma}

\begin{remark}\label{remmultiplesvaleurs}
Let $u \in \UU \cap \L^\infty([0,T],\U)$. Note that, for any Lebesgue point~$\t_q \in \mathcal{L}(f_u)$ considered in a multiple needle-like variation (see Definition~\ref{def_multineedle}), it is possible to consider several values~$\omega^r_q \in \U$ for~$r=1,\ldots,R_q$ with~$R_q \in \N^*$. This additional degree of freedom is essential in order to generate the $\U$-Pontryagin cone of~$u$ with multiple needle-like variations, as developed in the next remark.
\end{remark}

\begin{remark}\label{rempontconegenerated}
The $\U$-Pontryagin cone of a control~$u \in \UU \cap \L^\infty([0,T],\U)$ is generated by multiple needle-like variations as follows. Consider some~$z \in \Pont_\U[u]$. Definition~\ref{def_pontcone} gives
$$ z = \sum_{q=1}^{ \ttilde{Q} } \lambda_q w^u_{(\t_q,\omega_q)}(T) $$
for some~$\ttilde{Q} \in \N^*$, where~$\lambda_q \geq 0$ and~$(\t_q,\omega_q) \in \mathcal{L}(f_u) \times \U$ for all~$q \in \{ 1,\ldots,\ttilde{Q} \}$. By gathering the Lebesgue points~$\t_q$ that are equal (and thus gathering the corresponding values~$\omega_q$, see Remark~\ref{remmultiplesvaleurs}), we construct a package~$\chi = (\overline{\t},\overline{\omega}) \in \mathcal{L}(f_u)^Q \times \U^R$ as in Definition~\ref{def_multineedle} (with~$Q \leq R = \ttilde{Q}$) and 
$$ z = \sum_{q=1}^{Q} \sum_{r=1}^{R_q} \lambda^r_q w^u_{(\t_q,\omega^r_q)}(T). $$
Denoting by
$ \overline{\lambda} = \{ \lambda^r_q \}^{ r=1, \ldots, R_q}_{q=1,\ldots,Q} \in \R^R_+$, 
we introduce the~$\C^1$ mapping~$\Psi' : [0,\beta'] \to \R^n$, defined by~$\Psi'(\alpha) = \Psi(\alpha \overline{\lambda})$ for all~$\alpha \in [0,\beta']$, where~$\Psi$ is the mapping defined in Lemma~\ref{lem_multineedle} and where~$\beta ' > 0$ is sufficiently small to guarantee that~$\alpha \overline{\lambda} \in [0,\beta]^R$ for all~$\alpha \in [0,\beta ']$. We finally get that
$$ \lim\limits_{\alpha \to 0^+} \dfrac{ \E ( u^{ \alpha \overline{\lambda} }_\chi ) - \E ( u ) }{ \alpha } = z  $$
because~$\frac{\partial \Psi'}{\partial \alpha}(0) = \D \Psi (0_{\R^R}) \cdot \overline{\lambda} = z$.
\end{remark}

Now let us prove Proposition~\ref{prop_wUreginterior}. Let~$u \in \UU \cap \L^\infty([0,T],\U)$ be weakly~$\U$-regular. Thus~$\Pont_\U[u] = \R^n$ contains~$e_j$ and~$-e_j$ for all~$j \in \{ 1,\ldots,n \}$, where~$\{ e_j \}_{j=1,\ldots,n}$ is the canonical basis of~$\R^n$. For all~$j \in \{ 1,\ldots,n \}$, Definition~\ref{def_pontcone} gives
$$ e_j = \sum_{q=1}^{\ttilde{Q}^{j+}} \lambda^{j+}_q w^u_{(\t^{j+}_q,\omega^{j+}_q)}(T) $$
for some~$\ttilde{Q}^{j+} \in \N^*$, where~$\lambda^{j+}_q \geq 0$ and~$(\t^{j+}_q,\omega^{j+}_q) \in \mathcal{L}(f_u) \times \U$ for all~$q \in \{ 1,\ldots,\ttilde{Q}^{j+} \}$, and
$$ -e_j = \sum_{q=1}^{\ttilde{Q}^{j-}} \lambda^{j-}_q w^u_{(\t^{j-}_q,\omega^{j-}_q)}(T) $$
for some~$\ttilde{Q}^{j-} \in \N^*$, where~$\lambda^{j-}_q \geq 0$ and~$(\t^{j-}_q,\omega^{j-}_q) \in \mathcal{L}(f_u) \times \U$ for all~$q \in \{ 1,\ldots,\ttilde{Q}^{j-} \}$. By gathering the Lebesgue points~$\t^{j\pm}_q$ which are equal (and thus gathering the corresponding values~$\omega^{j\pm}_q$, see Remark~\ref{remmultiplesvaleurs}), we construct a package~$\chi = (\overline{\t},\overline{\omega}) \in \mathcal{L}(f_u)^Q \times \U^R$ as in Definition~\ref{def_multineedle} (with~$Q \leq R = \sum_{j=1}^n (\ttilde{Q}^{j+} + \ttilde{Q}^{j-})$). Considering the~$\C^1$ mapping~$\Psi$ defined in Lemma~\ref{lem_multineedle}, it is clear, in the same spirit as in Remark~\ref{rempontconegenerated}, that each vector~$e_j$ and~$-e_j$ belong to~$\D \Psi (0_{\R^R}) \cdot \R^R_+ $, and thus~$\D \Psi (0_{\R^R}) \cdot \R^R_+ = \R^n$. Now we define the mapping~$\Phi : \R^n \times [0,\beta]^{R} \longrightarrow \R^n $ by~$
\Phi(z,\bbar{\alpha}) =  \Psi(\bbar{\alpha}) - z $ for all~$(z,\bbar{\alpha}) \in \R^n \times [0,\beta]^{R}$. The mapping~$\Phi$ is of class~$\C^1$ and satisfies~$\Phi(x_u(T),0_{\R^{R}}) = 0_{\R^n}$ and~$\frac{\partial \Phi}{\partial \bbar{\alpha}} (x_u(T),0_{\R^{R}}) \cdot \R^{R}_+  = \D \Psi (0_{\R^R}) \cdot \R^R_+ = \R^n$. From the conic implicit function theorem~\cite[Theorem~1]{antoine1990}, there exists an open neighborhood~$\VV$ of~$x_u(T)$ and a continuous mapping~$\bbar{\alpha} : \VV \to [0,\beta]^{R}$ satisfying~$\bbar{\alpha} ( x_u(T) ) = 0_{\R^{R}}$ and~$\Phi(z , \bbar{\alpha} (z) ) = 0_{\R^n}$ for all~$z \in \VV$. Then it suffices to introduce the mapping~$V : \VV \to \UU \cap \L^\infty([0,T],\U)$ defined by~$V(z) = u^{\bbar{\alpha}(z)}_\chi $ for all~$z \in \VV$. By Remark~\ref{rembef}, the mapping~$V$ is continuous when endowing its codomain with the $\L^1$-metric.

\subsection{Proof of Proposition~\ref{prop_charactsUsing}}\label{secproof_charactsUsing}


\begin{lemma}\label{lemcharactstrongextremal}
Let $u \in \UU \cap \L^\infty([0,T],\U)$ and~$p \in \AC ([0,T],\R^n)$ be a solution to~\eqref{eqAE}. The following statements are equivalent:
\begin{enumerate}[label=\rm{(\roman*)}]
\item\label{lemcharactstrongextremal_i} $(x_u,u,p)$ is a strong $\U$-extremal lift of the pair~$(x_u,u)$;
\item\label{lemcharactstrongextremal_iii} $\langle p(T) ,z \rangle_{\R^n} \leq 0$ for all~$ z \in \Pont_\U[u]$;
\item\label{lemcharactstrongextremal_iv} $\langle p(T) , w^u_{(\t,\omega) }(T) \rangle_{\R^n} \leq 0$ for all~$(\t,\omega) \in \mathcal{L}(f_u) \times \U$.
\end{enumerate}
\end{lemma}

\begin{proof}
The equivalence between~\ref{lemcharactstrongextremal_iii} and~\ref{lemcharactstrongextremal_iv} follows from Definition~\ref{def_pontcone}. For all~$(\t,\omega) \in \mathcal{L}(f_u) \times \U$, we set~$h_{(\t,\omega)}(t) = \langle p(t) , w^u_{(\t,\omega)}(t) \rangle_{\R^n}$ for all~$t \in [\t,T]$ which is constant thanks to~\eqref{eqAE}. Note that~\ref{lemcharactstrongextremal_iv}, which can be written as~$ h_{(\t,\omega)}(T) \leq 0$ for all~$(\t,\omega) \in \mathcal{L}(f_u) \times \U$, is equivalent to~$h_{(\t,\omega)}(\t) \leq 0$ for all~$(\t,\omega) \in \mathcal{L}(f_u) \times \U$, which exactly corresponds to the Hamiltonian maximization condition~\eqref{eqHM}, giving~\ref{lemcharactstrongextremal_i}.
\end{proof}

Let us prove Proposition~\ref{prop_charactsUsing}. Let~$u \in \UU \cap \L^\infty([0,T],\U)$. First, assume that~$u$ is strongly~$\U$-singular, i.e., $ \Pont_\U[u]$ is a proper subcone of~$\R^n $. Hence~$0_{\R^n}$ belongs to its boundary and, since~$ \Pont_\U[u]$ is also convex, by a standard separation argument, there exists~$\psi \in \R^n \backslash \{ 0_{\R^n} \}$ such that~$\langle \psi , z \rangle_{\R^n}  \leq 0$ for all~$ z \in \Pont_\U[u]$. Considering~$p \in \AC([0,T],\R^n)$ the unique solution to~\eqref{eqAE} ending at~$p(T) = \psi$ (in particular~$p$ is not trivial), we obtain that~$\langle p(T) ,z \rangle_{\R^n}  \leq 0$ for all~$ z \in \Pont_\U[u]$. By Lemma~\ref{lemcharactstrongextremal}, $(x_u,u,p)$ is a nontrivial strong~$\U$-extremal lift of~$(x_u,u)$. Conversely, assume that~$u$ is weakly~$\U$-regular, i.e., $\Pont_\U[u] = \R^n$. By contradiction, let us assume that~$(x_u,u)$ admits a nontrivial strong~$\U$-extremal lift~$(x_u,u,p)$. Since~$p(T) \in \Pont_\U[u] = \R^n$, it follows from Lemma~\ref{lemcharactstrongextremal} that~$ \Vert p(T) \Vert_{\R^n}^2 \leq 0$ and thus~$p(T) = 0_{\R^n}$. Since the adjoint equation~\eqref{eqAE} is linear, it follows that~$p$ is trivial, which raises a contradiction. 

\subsection{Proof of Proposition~\ref{proplinear2a2b} (only the sufficient condition)}\label{secproof_proplinear2}

First step: assume that~$\U$ is convex and that~$x_u(T)$ belongs to the interior of the~$\L^\infty_\U$-accessible set. Let us prove that~$u$ is strongly~$\U$-regular (and thus is weakly~$\U$-regular by Proposition~\ref{proplinear}). By contradiction assume that~$u$ is weakly~$\U$-singular. By Proposition~\ref{prop_charactUsing}, let~$(x_u,u,p)$ be a nontrivial weak~$\U$-extremal lift of the pair~$(x_u,u)$. Since the adjoint equation~\eqref{eqAE} is linear, we know that~$p(T) \neq 0_{\R^n}$. Since~$x_u(T)$ belongs to the interior of~$\E ( \L^\infty([0,T],\U))$, there exist~$\gamma > 0$ sufficiently small and~$v \in \L^\infty([0,T],\U)$ such that $ x_u(T) + \gamma p(T) = \E(v) $. Since the control system~\eqref{contsys} is linear, $\E$ is affine and thus~$
 \gamma p(T) = \E(v) - \E(u) = \D \E (u) \cdot (v - u) $. Then~$ \gamma \Vert p(T) \Vert^2_{\R^n} = \langle p(T) , \D \E (u) \cdot (v - u) \rangle_{\R^n} \leq 0$ by Lemma~\ref{lemcharactweakUextremal}, and thus~$p(T) = 0_{\R^n}$, which raises a contradiction.
 
Second step: in the general control constraints case, assume that~$x_u(T)$ belongs to the interior of the~$\L^\infty_\U$-accessible set. Then~$x_u(T)$ belongs to the interior of the~$\L^\infty_{\mathrm{conv}(\U)}$-accessible set. Since~$u \in \L^\infty([0,T],\U) \subset \L^\infty([0,T],\mathrm{conv}(\U))$, we infer from the first step that~$u$ is strongly~$\mathrm{conv}(\U)$-regular. We deduce that~$u$ is weakly~$\U$-regular from Proposition~\ref{proplinear}.

\section{Proofs of results of Section~\ref{sec_comments}}\label{secproofs2}

\subsection{Proof of Proposition~\ref{propsensitive}}\label{secproof_propsensitive}

\begin{remark}\label{remsensitive}
Given a partition~$\T$ of~$[0,T]$, it is clear that a target point~$x^1 \in \R^n$ is~$\PC^\T_\U$-reachable in time~$T$ from~$x^0$ if and only if~$x^1$ is~$\PC^{\T'}_\U$-reachable in time~$T$ from~$x^0$ for at least one partition~$\T'$ of~$[0,T]$ such that~$\T' \subset \T$, if and only if~$x^1$ is~$\PC^{\T'}_\U$-reachable in time~$T$ from~$x^0$ for all partitions~$\T'$ of~$[0,T]$ such that~$\T \subset \T'$.
\end{remark}

Let~$u \in \UU \cap \L^\infty([0,T],\U)$ and assume that Property~\eqref{eqPu} is not satisfied. Let~$\T=\{ t_i \}_{i=0,\ldots,N}$ be a partition of~$[0,T]$ and~$\eps > 0$. Since Property~\eqref{eqPu} is not satisfied, there exists a partition~$\T ' = \{ t'_i \}_{i=0,\ldots,N'}$ of~$[0,T]$ such that~$\Vert \T' \Vert < 2\eps$ and such that~$x_u(T)$ is not~$\PC^{\T'}_\U$-reachable in time~$T$ from~$x^0$. For any~$i \in \{ 1,\ldots,N-1 \}$, the intersection~$ \T' \cap (t_i - \eps , t_i+\eps)$ is not empty and we select~$t^\eps_i$ one of its elements. For~$i=0$ (resp.\ $i=N$), we choose~$t^\eps_0 = 0$ (resp.\ $t^\eps_N = T$). Consider the partition~$\T^\eps = \{ t^\eps_i \}_{i=0,\ldots,N}$ of~$[0,T]$. Since~$\T^\eps \subset \T'$, we know from Remark~\ref{remsensitive} that~$x_u(T)$ is not~$\PC^{\T^\eps}_\U$-reachable in time~$T$ from~$x^0$. 

\subsection{Proof of Proposition~\ref{propn1}}\label{secproof_propn1}

\begin{lemma}[Approximated reachability]\label{lemproofpropn1}
Given any~$u \in \UU \cap \L^\infty([0,T],\U)$ and any~$\eps > 0$, there exists a threshold~$\delta > 0$ such that, for any partition~$\T$ of~$[0,T]$ satisfying~$\Vert \T \Vert \leq \delta$, there exists~$v \in \UU \cap \PC^\T([0,T],\U)$ such that~$\Vert x_v(T) - x_u(T) \Vert_{\R^n} \leq \eps$.
\end{lemma}

\begin{proof}
Let~$u \in \UU \cap \L^\infty([0,T],\U)$ and~$\eps > 0$. Take~$s=1$ and~$M = \Vert u \Vert_{\L^\infty}$ in Lemma~\ref{lemopenL1} and let~$L^M > 0$ being a positive Lipschitz constant of~$\E$ restricted to~$ \NNN_{\L^1}(u,\rho^M,M)$ endowed with the~$\L^1$-metric. By Proposition~\ref{proplusin}, there exists~$\delta > 0$ such that, for any partition~$\T$ of~$[0,T]$ satisfying~$\Vert \T \Vert \leq \delta$, there exists~$v \in \PC^\T([0,T],\U)$ such that~$\Vert v-u \Vert_{\L^1} \leq \min ( \rho^M , \frac{\eps}{L^M} )$ and~$\Vert v \Vert_{\L^\infty} \leq \Vert u \Vert_{\L^\infty} = M$. Since~$v \in \NNN_{\L^1}(u,\rho^M,M) \subset \UU$, from Lemma~\ref{lemopenL1}, we have~$\Vert x_v(T) - x_u(T) \Vert_{\R^n} = \Vert \E(v) - \E(u) \Vert_{\R^n} \leq L^M \Vert v - u \Vert_{\L^1} \leq \eps$. 
\end{proof}

Let us prove Proposition~\ref{propn1}. Let~$u \in \L^\infty([0,T],\U)$ be such that~$x_u(T)$ belongs to the interior of the~$\L^\infty_\U$-accessible set. There exist~$u'$, $u'' \in \L^\infty([0,T],\U)$ such that~$x_{u'}(T) < x_u(T) < x_{u''} (T)$. We infer from Lemma~\ref{lemproofpropn1} that there exists~$\delta  > 0$ such that, for any partition~$\T$ of~$[0,T]$ satisfying~$\Vert \T \Vert \leq \delta $, there exist~$v'$, $v'' \in \PC^\T([0,T],\U)$ such that~$x_{v'}(T) \leq x_u(T) \leq x_{v''}(T)$. Now let us fix such a partition~$\T$ of~$[0,T]$ which satisfies~$\Vert \T \Vert \leq \delta$. In view of the above, we know that~$x_u(T)$ belongs to the convex hull of~$\E ( \PC^\T([0,T],\U))$. On the other hand, since~$\U$ is convex, $\PC^\T([0,T],\U)$ is convex and thus is a connected set. Since~$\E$ is continuous on~$\UU = \L^\infty([0,T],\R^m)$, we deduce that~$\E ( \PC^\T([0,T],\U))$ is a connected set of~$\R$, and thus is convex. We have proved that~$x_u(T) \in \E ( \PC^\T([0,T],\U))$.

\subsection{Proof of Theorem~\ref{thmmain} under strong~$\U$-regularity}\label{sec_firstproof}

Let $u \in \mathcal{U} \cap \L^\infty([0,T],\U)$ be a control such that~$x^1 = x_u(T)= \E (u)$. Let~$M = \Vert x_u \Vert_{\C} +\Vert u \Vert_{\L^\infty} + 1$ and let us fix some~$1 < s < +\infty$. Using the truncated dynamics~$f^M$ introduced in Appendix~\ref{apptruncated}, we have~$x^M_u = x_u $ and~$ \D\E^M (u) = \D\E(u) $ (see Remark~\ref{remtruncated}). Assume that~$u$ is strongly $\U$-regular. By Definition~\ref{def_Ureg}, there exists a~$2n$-tuple~$\bbar{v} = \{ v_j  \}_{j=1,\ldots,2n}$ of elements of~$\TT_{\L^\infty_\U}[u]$ such that 
\begin{equation}\label{condtfic}
\begin{split}
& \D\E^M(u) \cdot v_j = \D\E(u)  \cdot v_j = e_j  \\
& \D\E^M(u)  \cdot v_{n+j} = \D\E(u)  \cdot v_{n+j} = -e_j 
\end{split}
\end{equation}
for every~$j \in \{ 1,\ldots,n \}$, where~$\{ e_j \}_{j=1,\ldots,n}$ is the canonical basis of~$\R^n$. We define the mapping~$\Psi: \L^s([0,T],\R^m) \times \L^s([0,T],\R^m)^{2n} \times \R^{2n}_+\longrightarrow \R^n $ by
$$
\Psi(y,\bbar{z},\bbar{\alpha}) =  \E^M \bigg( y + \di \sum_{j=1}^{2n} \alpha_j z_j \bigg)  
$$
for all~$(y,\bbar{z},\bbar{\alpha}) \in \L^s([0,T],\R^m) \times \L^s([0,T],\R^m)^{2n} \times \R^{2n}_+$. This mapping satisfies~$\Psi(u,\bbar{v},0_{\R^{2n}}) = \E^M (u) = x^M_u(T) = x_u(T) = x^1$. Furthermore, since~$\E^M : \L^s([0,T],\R^m) \to \R^n$ is of class~$\C^1$ (see Proposition~\ref{propinputoutputR}), the mapping~$\Psi$ is also of class~$\C^1$ and we infer from~\eqref{condtfic} that~$
\frac{\partial\Psi}{\partial\bbar{\alpha}}(u,\bbar{v},0_{\R^{2n}}) \cdot \R^{2n}_+ = \R^n$. By the conic implicit function theorem~\cite[Theorem~1]{antoine1990}, there exists a continuous mapping~$\bbar{\alpha} : \BB_{\L^s} (u,\eta) \times \BB_{(\L^s)^{2n}} (\bbar{v},\eta) \to \R^{2n}_+$, with~$\eta > 0$, satisfying~$\bbar{\alpha} (u,\bbar{v}) = 0_{\R^{2n}}$ and~$\Psi(y,\bbar{z},\bbar{\alpha}(y,\bbar{z})) = x^1$ for all~$(y,\bbar{z}) \in \BB_{\L^s} (u,\eta) \times \BB_{(\L^s)^{2n}} (\bbar{v},\eta)$. 

By Lemma~\ref{lemaverage2}, there exists a threshold~$\delta > 0$ such that~$\II^\T(u) \in \BB_{\L^s} (u,\eta)$ and~$\II^\T(\bbar{v}) \in \BB_{(\L^s)^{2n}} (\bbar{v},\eta)$, and thus
$$ \Psi \Big( \II^\T(u),\II^\T(\bbar{v}),\bbar{\alpha} \big( \II^\T(u),\II^\T(\bbar{v}) \big) \Big) = x^1 $$
for any partition~$\T$ of $[0,T]$ satisfying~$\Vert \T \Vert \leq \delta$, where~$\II^\T$ is the averaging operator introduced in Appendix~\ref{app3}. For any partition~$\T $ of~$[0,T]$ satisfying~$\Vert \T \Vert \leq \delta$, we define the control
$$ V^\T = u + \di \sum_{j=1}^{2n} \alpha_j (\II^\T(u),\II^\T(\bbar{v})) v_j \in \L^{\infty}([0,T],\R^m). $$
Using the linearity of the averaging operators, we obtain the piecewise constant control
\begin{multline*}
\II^\T(V^\T) = \II^\T(u) \\
+ \di \sum_{j=1}^{2n} \alpha_j (\II^\T(u),\II^\T(\bbar{v})) \II^\T (v_j) \in \PC^\T([0,T],\R^m)
\end{multline*}
which satisfies
$$ \E^M ( \II^\T(V^\T) ) = \Psi \Big( \II^\T(u),\II^\T(\bbar{v}),\bbar{\alpha} \big( \II^\T(u),\II^\T(\bbar{v}) \big) \Big) = x^1  $$
for all partitions~$\T $ of~$[0,T]$ satisfying~$\Vert \T \Vert \leq \delta$. If necessary we take a smaller value of~$\delta > 0$ to have~$\Vert \II^\T(u) - u \Vert_{\L^s}$ and~$\Vert \II^\T(\bbar{v})) - \bbar{v} \Vert_{(\L^s)^{2n}}$ small enough (by Lemma~\ref{lemaverage2}), and thus~$\Vert \bbar{\alpha} (\II^\T(u),\II^\T(\bbar{v})) \Vert_{\R^{2n}}$ small enough as well, to get that:
\begin{enumerate}[label=\rm{(\roman*)}]
\item $\Vert \II^\T(V^\T) \Vert_{\L^\infty} \leq \Vert V^\T \Vert_{\L^\infty} \leq \Vert u \Vert_{\L^\infty} + 1 \leq M$ (here we used in particular Lemma~\ref{lemaverage1}); 
\item $\Vert \II^\T(V^\T) - u \Vert_{\L^s} \leq \Vert \II^\T(V^\T) - \II^\T(u) \Vert_{\L^s} + \Vert \II^\T(u) - u \Vert_{\L^s} \leq \Vert V^\T - u \Vert_{\L^s} + \Vert \II^\T(u) - u \Vert_{\L^s} \leq \rho^M$ where~$\rho^M > 0$ is given in Lemma~\ref{lemopenL1} (here also we used Lemma~\ref{lemaverage1});
\item $V^\T$ is with values in~$\U$ (which is possible by Lemma~\ref{lemconetangent} with~$J = 2n$ and using that~$v_j \in \TT_{\L^\infty_\U}[u]$ for all~$j \in \{ 1,\ldots,2n \}$), and thus so is $\II^\T (V^\T)$ by Proposition~\ref{propaverage3};
\end{enumerate}
for all partitions~$\T$ of~$[0,T]$ satisfying $\Vert \T \Vert \leq \delta$. 

We are now in a position to conclude the proof. Let us fix a partition~$\T$ of~$[0,T]$ satisfying~$\Vert \T \Vert \leq \delta$ and, for the ease of notations, let us denote simply by~$V = \II^\T (V^\T) \in \PC^\T([0,T],\R^m)$ and recall that~$\E^M(V)=x^1$. Since~$ V$ is with values in~$\U$ from the above item~(iii), we have~$V  \in \PC^\T([0,T],\U)$. By the above items~(i) and~(ii) and by Lemma~\ref{lemopenL1}, we have~$V \in \NNN_{\L^s} ( u , \rho^M , M ) \subset \mathcal{U}$ and~$\Vert x_{V} - x_u \Vert_{\C} \leq 1$. We infer that~$\Vert x_{V} \Vert_{\C} \leq \Vert x_u \Vert_{\C} + 1 \leq M$ and, since~$\Vert V \Vert_{\L^\infty} \leq M$ from the above item~(i), we obtain from Remark~\ref{remtruncated} that~$x^M_{V} = x_{V}$ and thus~$ \E (V) = x_{V}(T) = x^M_{V}(T) = \E^M (V) = x^1$. The proof is complete.

\subsection{Proof of Theorem~\ref{thmmain} under weak~$\U$-regularity}\label{sec_secondproof}

Let $u \in \mathcal{U} \cap \L^\infty([0,T],\U)$ be a control such that~$x^1 = x_u(T)= \E (u)$. Assume that~$u$ is weakly~$\U$-regular and, by contradiction, that Property~\eqref{eqPu} is not satisfied. Then there exists a sequence~$(\T_k)_{k \in \N}$ of partitions of~$[0,T]$ such that~$\Vert \T_k \Vert \to 0$ as~$k \to +\infty$ and such that~$x^1$ is not~$\PC^{\T_k}_\U$-reachable in time~$T$ from~$x^0$ for all~$k \in \N$.

We first introduce several notations. Since~$u$ is weakly~$\U$-regular, considering~$\{ e_j \}_{j=1,\ldots,n}$ the canonical basis of~$\R^n$, we construct a package~$\chi = (\overline{\t},\overline{\omega}) \in \mathcal{L}(f_u)^Q \times \U^R$ as in the proof of Proposition~\ref{prop_wUreginterior}. Now take~$s=1$ and~$M =  \Vert u \Vert_{\L^\infty} + \Vert \bbar{\omega} \Vert_{(\R^m)^R}$ and consider~$\rho^M > 0$ given in Lemma~\ref{lemopenL1}. As in Remark~\ref{rembef}, there exists~$\beta > 0$ sufficiently small so that~$u^{\bbar{\alpha}}_\chi \in \NNN_{\L^1}(u,\frac{\rho^M}{2},M) $ for all~$\bbar{\alpha} \in [0,\beta]^R$. In particular we have~$u^{\bbar{\alpha}}_\chi \in \UU \cap \L^\infty([0,T],\U) $ for all~$\bbar{\alpha} \in [0,\beta]^R$. Consider the~$\C^1$ mapping~$\Psi : [0,\beta]^R \to \R^n$, defined by~$\Psi(\bbar{\alpha}) = \E ( u^{\bbar{\alpha}}_\chi )$ for all~$\bbar{\alpha} \in [0,\beta]^R$, which satisfies~$\Psi(0_{\R^R}) = x_u(T)$ and~$\D \Psi (0_{\R^R})\cdot \R^R_+ = \R^n$ as in the proof of Proposition~\ref{prop_wUreginterior}.

We define the~$\C^1$ mapping~$\Phi : \R^n \times [0,\beta]^R \to \R^n$ by~$\Phi(z,\bbar{\alpha}) = \Psi(\bbar{\alpha}) - z$ for all~$(z,\bbar{\alpha}) \in \R^n \times [0,\beta]^R$. It follows from the above arguments that~$\frac{\partial \Phi}{\partial \bbar{\alpha}} (x_u(T),0_{\R^R}) \cdot \R^R_+ = \R^n$ and, since~$\Phi( x_u(T) , 0_{\R^R}) = 0_{\R^n}$, the conic implicit function theorem~\cite[Theorem~1]{antoine1990} provides the existence of a continuous mapping~$\bbar{\alpha} : \BB_{\R^n}(x_u(T),\eta) \to [0,\beta]^R$, with~$\eta > 0$, such that~$\bbar{\alpha}(x_u(T)) = 0_{\R^R}$ and~$\Phi ( z , \bbar{\alpha}(z) ) = 0_{\R^n}$ for all~$z \in \BB_{\R^n}(x_u(T),\eta)$. 

The mapping~$V : \BB_{\R^n} (x_u(T),\eta) \to \NNN_{\L^1}(u,\frac{\rho^M}{2},M)$, defined by~$V(z) = u^{\bbar{\alpha}(z)}_\chi $ for all~$z \in \BB_{\R^n} (x_u(T),\eta)$, is such that~$V(z) \in \UU \cap \L^\infty([0,T],\U) $ for all~$z \in \BB_{\R^n} (x_u(T),\eta)$. When endowing the codomain with the~$\L^1$-metric, the continuity of~$V$ follows from the continuity of~$\bbar{\alpha}$ and from Remark~\ref{rembef}. Finally note that~$x_{V(z)}(T) = \E ( V(z) ) = \E ( u^{\bbar{\alpha}(z)}_\chi ) = \Psi ( \bbar{\alpha}(z) ) 
= \Phi ( z , \bbar{\alpha}(z) ) + z =  z$
for all~$z \in \BB_{\R^n} (x_u(T),\eta)$.

In what follows we denote by~$L^M > 0$ a positive Lipschitz constant of~$\E$ restricted to~$\NNN_{\L^1}(u,\rho^M,M)$ endowed with the~$\L^1$-metric (see Lemma~\ref{lemopenL1}). By contradiction, assume that, for all~$k \in \N$, there exists some~$z_k \in  \BB_{\R^n} (x_u(T),\eta)$ such that
$$ 
\min \bigg( \frac{\rho^M}{2} , \frac{\eta}{L^M} \bigg) < \Vert V(z_k) - \II^{\T_k} ( V(z_k) )  \Vert_{\L^1}  ,
$$ 
where~$\II^{\T_k}$ is the averaging operator introduced in Appendix~\ref{app3}. By compactness of~$\BB_{\R^n} (x_u(T),\eta)$, up to a subsequence (that we do not relabel), the sequence~$(z_k)_{k \in \N}$ converges to some~$z' \in \BB_{\R^n} (x_u(T),\eta)$. We infer from Lemma~\ref{lemaverage1} that
\begin{multline*}
 \min \bigg( \frac{\rho^M}{2} , \frac{\eta}{L^M} \bigg) < 2 \Vert V(z_k) - V(z') \Vert_{\L^1} \\
 + \Vert V( z' ) - \II^{\T_k} ( V( z' ) )  \Vert_{\L^1} 
\end{multline*}
for every~$k \in \N$, raising a contradiction when~$k \to + \infty$ by continuity of~$V$ and by Lemma~\ref{lemaverage2}. We conclude that there exists $K \in \N$ such that
\begin{equation}\label{ineq765}
\Vert V(z) - \II^{\T_K} ( V(z) )  \Vert_{\L^1} \leq \min \left( \frac{\rho^M}{2} , \frac{\eta}{L^M} \right)
\end{equation}
for every~$z \in \BB_{\R^n} (x_u(T),\eta)$. Since~$V(z) \in \NNN_{\L^1}(u,\frac{\rho^M}{2},M)$, we deduce from~\eqref{ineq765} and from Lemma~\ref{lemaverage1} that~$ \II^{\T_K} ( V(z) )  \in \NNN_{\L^1} ( u,\rho^M,M) $ for all~$z \in \BB_{\R^n} (x_u(T),\eta)$. Since~$V(z) \in \L^\infty([0,T],\U) $, we infer from Proposition~\ref{propaverage3} that~$ \II^{\T_K} ( V(z) )  \in \PC^{\T_K}([0,T],\U) $ for all~$z \in \BB_{\R^n} (x_u(T),\eta)$.

To conclude the proof of Theorem~\ref{thmmain}, we define $\BBB : \BB_{\R^n} (x_u(T),\eta) \to \R^n$ by
\begin{multline*}
\BBB(z) =  x_u(T) + z - x_{ \II^{\T_K} ( V(z) ) } (T)  \\
= \E (u) + \E \Big( V(z) \Big) - \E \Big( \II^{\T_K} ( V(z) ) \Big)
\end{multline*}
for every~$z \in \BB_{\R^n} (x_u(T),\eta)$. By Lemma~\ref{lemaverage1} and thanks to the continuities of the mapping~$V$ and of the restriction of~$\E$ on~$ \NNN_{\L^1} ( u,\rho^M,M)$ endowed with the~$\L^1$-metric, $\BBB$ is continuous. Furthermore, since~$V(z)$ and~$\II^{\T_K} ( V(z) )$ both belong to~$\NNN_{\L^1} ( u,\rho^M,M) $, we have
\begin{multline*}
\Vert \BBB (z) - x_u(T) \Vert_{\R^n} =  \left\Vert \E \Big( V(z) \Big) - \E \Big( \II^{\T_K} ( V(z) ) \Big) \right\Vert_{\R^n}  \\
\leq  L^M \Vert V(z) - \II^{\T_K} ( V(z) ) \Vert_{\L^1}    \leq \eta 
\end{multline*}
for every~$z \in \BB_{\R^n} (x_u(T),\eta)$, where we have used~\eqref{ineq765}. Therefore~$\BBB$ is a continuous mapping from~$\BB_{\R^n} (x_u(T),\eta)$ with values in~$\BB_{\R^n} (x_u(T),\eta)$. By the Brouwer fixed-point theorem,~$\BBB$ has a fixed-point~$z^* \in \BB_{\R^n} (x_u(T),\eta)$, and thus
$$  x_{ \II^{\T_K} ( V(z^*) ) } (T) = x_u(T) = x^1 .$$
Since~$\II^{\T_K} ( V(z^*) ) \in \UU \cap \PC^{\T_K}([0,T],\U)$, $x^1$ is~$\PC^{\T_K}_\U$-reachable in time~$T$ from~$x^0$, raising a contradiction. 

\appendix

\subsection{An example}\label{appexamplegrasse}
We develop here an example inspired from~\cite[Section~II]{grasse1995}, showing that the converse of the geometric Pontryagin maximum principle is not true in general and that, given a control~$u \in \UU \cap \L^\infty([0,T],\U)$, the condition that~$x_u(T)$ belongs to the interior of the~$\L^\infty_\U$-accessible set is not a sufficient condition for Property~\eqref{eqPu}, even if~$\U$ is convex.

Take~$T = n = m = 2$ and~$\U = \R^2$. Take~$g_1 \in \C([0,2],\R)$ be a continuous function that is positive on the interval~$[0,1)$ and vanishing on the interval~$[1,2]$. Take~$g_2 \in \L^\infty([0,2],\R)$ be arbitrarily fixed and~$g_3 \in \AC([0,2],\R)$ be defined by~$g_3(t) = \int_0^t g_1(\xi)g_2(\xi) \, d\xi$ for all~$t \in [0,2]$. Note that~$g_3$ is constant on the interval~$[1,2]$. We denote by~$G$ the corresponding constant values. We set $x^0 = 0_{\R^2}$ and the expression of~$f((x_1,x_2),(u_1,u_2),t)$ by
$$  
\begin{pmatrix}
g_1(t)u_1 + g_1(2-t) \Big( (x_1-G)^2 + x_2^2 \Big) u_1 \\
\Big(x_1 - g_3(t) \Big)^2 + g_1(2-t) \Big( (x_1-G)^2 + x_2^2 \Big) u_2
\end{pmatrix}
$$
for all~$((x_1,x_2),(u_1,u_2),t) \in \R^2 \times \R^2 \times [0,2]$.

\begin{claim}\label{claim1}
The point~$(G,0)$ is an equilibrium of the control system on the interval~$[1,2]$, independently of the control.
\end{claim}

\begin{proof}
Since~$g_1(t) = 0$ and~$g_3(t) = G$ for all~$t \in [1,2]$, we have~$f((G,0),u,t) = 0_{\R^2}$ for all~$(u,t) \in \R^2 \times [1,2]$.
\end{proof}

\begin{claim}\label{claim2}
Let~$u \in \L^\infty([0,2],\R^2)$ satisfying~$u_1(t)=g_2(t)$ for a.e.\ $t \in [0,1]$. Then~$u \in \UU$ and~$x_u = (g_3,0)$. In particular~$x_u(2) = (G,0)$.
\end{claim}

\begin{proof}
Since~$g_1(2-\xi) = 0$ for all~$\xi \in [0,1]$, it holds that~$x_{u,1}(t) = \int_0^t g_1(\xi)u_1(\xi)\, d\xi = \int_0^t g_1(\xi)g_2(\xi)\, d\xi = g_3(t)$ for all~$t \in [0,1]$. From the second coordinate, we obtain that~$x_{u,2}(t) = 0$ for all~$t \in [0,1]$. Since~$x_u(1) = (g_3(1),0) = (G,0)$, we get from Claim~\ref{claim1} that~$x_u(t) = (G,0) = (g_3(t),0) $ for all~$t \in [1,2]$.
\end{proof}

\begin{claim}\label{claim3}
Let~$u \in \UU$ such that~$x_u(T') = (G,0)$ for some~$T' \in [1,2]$. Then~$u_1(t) = g_2(t)$ for a.e.\ $t \in [0,1]$.
\end{claim}

\begin{proof}
By Claim~\ref{claim1}, $x_u(1) = (G,0)$. Since~$g_1(2-\xi) = 0$ for all~$\xi \in [0,1]$, we get that~$0 = x_{u,2}(1) = \int_0^1 (x_{u,1}(\xi) - g_3(\xi) )^2 \, d\xi$ and thus~$x_{u,1}(t) = g_3(t)$ for all~$t \in [0,1]$. Derivating this equality leads to~$g_1(t)u_1(t) = g_1(t)g_2(t)$ for a.e.\ $t \in [0,1]$. Since~$g_1$ is positive on the interval~$[0,1)$, we get that~$u_1(t) = g_2(t)$ for a.e.\ $t \in [0,1]$.
\end{proof}

\begin{claim}\label{claim4}
The end-point mapping is surjective.
\end{claim}

\begin{proof}
Let~$x^1 \in \R^2$. Let us prove that there exists~$u \in \UU$ such that~$\E(u) = x_u(2) = x^1$. If~$x^1 = (G,0)$, from Claim~\ref{claim2}, it is sufficient to take any control~$u \in \L^\infty([0,2],\R^2)$ which satisfies~$u_1(t) = g_2(t)$ for a.e.\ $t \in [0,1]$. In the rest of this proof, we focus on the case~$x^1 \neq (G,0)$. 

Consider a function~$g_4 \in \L^\infty([0,\frac{3}{2}],\R)$ such that the measure of~$\{ t \in [0,1] \mid g_4(t) \neq g_2 (t) \}$ is positive and such that the~$\L^\infty$-norm of $g_4 - g_2$ on~$[0,1]$ is small enough to guarantee that any control~$u \in \L^\infty([0,2],\R^2)$ which satisfies~$u_1(t) = g_4(t)$ for a.e.\ $t \in [0,\frac{3}{2}]$ is admissible, i.e., $u \in \UU$. This is possible by Claim~\ref{claim2}, since~$\UU$ is an open subset of~$\L^\infty([0,2],\R^2)$. Take such a control~$u$ (which is only determined on the interval~$[0,\frac{3}{2}]$ at this step). By Claim~\ref{claim3}, $x_u( \frac{3}{2} ) \neq (G,0)$. Consider now a~$\C^1$ function~$\varrho : [\frac{3}{2},2] \to \R^2$ which satisfies~$\varrho(\frac{3}{2}) = x_u( \frac{3}{2} )$, $\varrho(2) = x^1$ and~$\varrho(t) \neq (G,0)$ for all~$t \in [\frac{3}{2},2]$. We determine the control~$u$ on~$[0,\frac{3}{2}]$ as
$$ 
u_1 (t) = \dfrac{\dot{\varrho_1}(t)}{\bbar{\varrho}(t)}, \quad u_2 (t) = \dfrac{\dot{\varrho_2}(t) - ( \varrho_1(t) - g_3(t) )^2 }{\bbar{\varrho}(t)},
$$
where~$\bbar{\varrho}(t)=g_1(2-t) ( (\varrho_1(t) - G)^2 + \varrho_2(t)^2 )$ for a.e.\ $t \in [\frac{3}{2},2]$. The control~$u$ belongs to~$\L^\infty([0,2],\R^2)$ and~$x_u  = \varrho$ along~$[\frac{3}{2},2]$. Thus~$\E (u) = x_u(2) = \varrho (2) = x^1$. 
\end{proof}

Let us prove that the converse of the geometric Pontryagin maximum principle is not true in general. Take a control~$u \in \L^\infty([0,T],\R^2)$ which satisfies~$u_1(t)=g_2(t)$ for a.e.\ $t \in [0,1]$. By Claims~\ref{claim2} and~\ref{claim4}, we have~$u \in \UU$ and~$x_u(2)$ belongs to the interior of the~$\L^\infty_{\R^2}$-accessible set. Consider the constant function~$p : [0,2] \to \R^2$ defined by~$p(t)=(0,1) \neq 0_{\R^2}$ for all~$t \in [0,2]$. One can easily check that~$(x_u,u,p)$ is a nontrivial strong $\R^2$-extremal lift of~$(x_u,u)$ and thus~$u$ is strongly~$\R^2$-singular by Proposition~\ref{prop_charactsUsing}.

We now prove that, given a control~$u \in \UU \cap \L^\infty([0,T],\U)$, the condition that~$x_u(T)$ belongs to the interior of the~$\L^\infty_\U$-accessible set is not a sufficient condition for Property~\eqref{eqPu}, even if~$\U$ is convex. Take~$g_2(t)=t$ for a.e.\ $t \in [0,1]$ (which is not piecewise constant). Even if~$(G,0)$ belongs to the interior of the~$\L^\infty_{\U}$-accessible set (Claim~\ref{claim4}), we easily infer from Claim~\ref{claim3} that~$(G,0)$ is not~$\PC^\T_{\U}$-reachable in time~$T$ from~$x^0$ for any partition~$\T$ of~$[0,T]$. Hence Property~\eqref{eqQu} is not satisfied, and neither is the stronger Property~\eqref{eqPu}.

\subsection{A general result on $\L^s$-approximation by piecewise constant functions}\label{app2}


\begin{proposition}\label{proplusin}
Let $1 \leq s < +\infty$. Given any~$u \in \L^\infty([0,T],\U)$ and any~$\eps > 0$, there exists a threshold~$\delta > 0$ such that, for any partition~$\T$ of~$[0,T]$ satisfying $\Vert \T \Vert \leq \delta$, there exists~$v \in \PC^\T([0,T],\U)$ such that~$\Vert v - u \Vert_{\L^s} \leq \eps$ and~$\Vert v \Vert_{\L^\infty} \leq \Vert u \Vert_{\L^\infty}$.
\end{proposition}

\begin{proof}
Let~$u \in \L^\infty([0,T],\U)$ and $\eps > 0$. By the Lusin theorem~\cite{lusin1912}, there exists a compact subset~$\K_\eps \subset [0,T]$ such that~$ ( 2 \Vert u \Vert_{\L^\infty} )^s \mu ( [0,T] \bs \K_\eps ) \leq \eps^s/2$, where~$\mu$ is the Lebesgue measure, and such that~$u$ is continuous on~$\K_\eps$. By uniform continuity of~$u$ on~$\K_\eps$, there exists~$\delta > 0$ such that~$ \Vert u(\xi_2) - u(\xi_1) \Vert_{\R^m} \leq \frac{\eps}{(2T)^{1/s}} $ for all~$\xi_1$, $\xi_2 \in \K_\eps$ satisfying~$\vert \xi_2 - \xi_1 \vert \leq \delta$. Now, let~$\T = \{ t_i \}_{i=0,\ldots,N}$ be a partition of~$[0,T]$ such that~$\Vert \T \Vert \leq \delta$. We set
$$ I = \{ i \in \{ 0,\ldots,N-1 \} \mid \mu (\K_\eps \cap [t_i,t_{i+1})  > 0 \}. $$
For every~$i \in I$, we consider some~$\xi_i \in \K_\eps \cap [t_i,t_{i+1})$ such that~$u(\xi_i) \in \U$ and~$\Vert u(\xi_i) \Vert_{\R^m} \leq \Vert u \Vert_{\L^\infty}$. We also consider some~$\omega \in \U$ such that~$\Vert \omega \Vert_{\R^m} \leq \Vert u \Vert_{\L^\infty}$. We now define
$$  v(t) = \left\lbrace \begin{array}{lcl}
u(\xi_i) & \text{if} & t \in [t_i,t_{i+1}) \text{ with } i \in I, \\
\omega & \text{if} & t \in [t_i,t_{i+1}) \text{ with } i \notin I,
\end{array}
\right. $$
for every~$t \in [0,T]$. In particular we have~$v \in \PC^\T([0,T],\U)$ and~$\Vert v \Vert_{\L^\infty} \leq \Vert u \Vert_{\L^\infty}$. Finally we get that
\begin{eqnarray*}
\Vert v - u \Vert^s_{\L^s} & = & \int_{[0,T] \bs \K_\eps} \Vert v(t) - u(t) \Vert_{\R^m}^s \; dt \\[3pt]
& & + \displaystyle \sum_{i=0}^{N-1} \int_{\K_\eps \cap [t_i,t_{i+1})} \Vert v(t) - u(t) \Vert_{\R^m}^s \; dt \\[3pt]
& \leq & (2 \Vert u \Vert_{\L^\infty})^s \mu ( [0,T] \bs \K_\eps ) \\[3pt]
& & + \displaystyle \sum_{i \in I} \int_{\K_\eps \cap [t_i,t_{i+1})} \Vert u(\xi_i) - u(t) \Vert_{\R^m}^s \; dt \\[3pt]
& \leq & \dfrac{\eps^s}{2} + \dfrac{\eps^s}{2T} \displaystyle \sum_{i \in I} \mu ( \K_\eps \cap [t_i,t_{i+1}) )  \leq \eps^s ,
\end{eqnarray*}
which concludes the proof.
\end{proof}

Note that Proposition~\ref{proplusin} is not true with~$s=+\infty$, as shown in the following Fuller-type example~\cite{Fuller1960}.

\begin{example}\label{exfuller}
Take~$T = 1$, $m=1$ and~$\U = \R$. Consider the oscillating function~$u \in \L^\infty([0,T],\U)$ defined by~$u(t)=1$ for a.e.\ $t \in (\frac{1}{k+1},\frac{1}{k}]$ for all even~$k \in \N^*$ and~$u(t)=0$ for a.e.\ $t \in (\frac{1}{k+1},\frac{1}{k}]$ for all odd~$k \in \N^*$. We have~$\Vert v - u \Vert_{\L^\infty} \geq \frac{1}{2}$ for all~$v \in \PC^\T([0,T],\U)$ and all partitions~$\T$ of~$[0,T]$. 
\end{example}

\begin{corollary}
Let~$1 \leq s < +\infty$. Given any~$u \in \L^s([0,T],\U)$ and any~$\eps > 0$, there exists a threshold~$\delta > 0$ such that, for any partition~$\T$ of~$[0,T]$ satisfying~$\Vert \T \Vert \leq \delta$, there exists~$v \in \PC^\T([0,T],\U)$ such that~$\Vert v - u \Vert_{\L^s} \leq \eps$.
\end{corollary}

\begin{proof}
Let~$u \in \L^s([0,T],\U)$ and~$\eps > 0$. We fix some~$\omega \in \U$ and we define~$ C_k = \{ t \in [0,T] \mid \Vert u(t) \Vert_{\R^m} \geq k \} $ and
$$ u_k (t) = \left\lbrace 
\begin{array}{lcl}
u(t) & \text{if} & t \notin C_k, \\
\omega  & \text{if} & t \in C_k,
\end{array}
\right. $$
for a.e.\ $t \in [0,T]$ and for every~$k \in \N$. In particular~$u_k \in \L^\infty([0,T],\U)$ for every $k \in \N$. It is clear that~$( u_k(t) - u(t)  )_{k \in \N}$ converges to~$0_{\R^m}$ as~$k \to +\infty$ and that~$\Vert u_k(t) - u(t) \Vert_{\R^m} \leq \Vert \omega \Vert_{\R^m} + \Vert u(t) \Vert_{\R^m}$ for a.e.\ $t \in [0,T]$. By the Lebesgue dominated convergence theorem, we get that~$\Vert u_k - u \Vert_{\L^s} \rightarrow 0$ as~$k \to +\infty$. Hence, there exists~$k \in \N$ such that~$\Vert u_k - u \Vert_{\L^s} \leq \frac{\eps}{2}$. By Proposition~\ref{proplusin}, there exists~$\delta > 0$ such that, for any partition~$\T$ of $[0,T]$ satisfying~$\Vert \T \Vert \leq \delta$, there exists~$v \in \PC^\T([0,T],\U)$ such that~$\Vert v - u_k \Vert_{\L^s} \leq \frac{\eps}{2}$ and thus~$ \Vert v - u \Vert_{\L^s} \leq \Vert v - u_k \Vert_{\L^s} + \Vert u_k - u \Vert_{\L^s} \leq \eps $.
\end{proof}

\subsection{Averaging operators}\label{app3}
\noindent For any partition~$\T = \{ t_i \}_{i=0,\ldots,N}$ of~$[0,T]$, we define the \textit{averaging operator}~$\II^\T : \L^1([0,T],\R^m) \to \PC^\T([0,T],\R^m)$ by
\begin{equation}\label{defaveraging}
\II^\T(u)(t) = \dfrac{1}{t_{i+1} -t_i} \int_{t_i}^{t_{i+1}} u(\xi) \, d\xi 
\end{equation}
for every~$t \in [t_i,t_{i+1})$, every~$i \in \{ 0,\ldots,N-1 \}$ and every~$u \in \L^1([0,T],\R^m) $. The aim of this section is to establish several useful properties of the averaging operators. 

Let~$\T= \{ t_i \}_{i=0,\ldots,N}$ be a partition of~$[0,T]$. The averaging operator~$\II^\T$ is linear and projects any integrable function onto a piecewise constant function respecting the partition~$\T$ (by averaging its value on each sampling interval~$[t_i,t_{i+1})$). Furthermore we have
\begin{equation}\label{eqdom}
\Vert \II^\T(u)(t) \Vert_{\R^m} \leq \Vert u(t) \Vert_{\R^m}
\end{equation}
for a.e.\ $t \in [0,T]$ and all~$u \in \L^1([0,T],\R^m) $. 

\begin{lemma}\label{lemaverage1}
Let~$1 \leq s \leq +\infty$. For any partition~$\T$ of~$[0,T]$, we have~$\Vert \II^\T(u) \Vert_{\L^s} \leq \Vert u \Vert_{\L^s}$ for all~$u \in \L^s([0,T],\R^m)$.
\end{lemma}

\begin{proof}
Let~$\T = \{ t_i \}_{i=0,\ldots,N} $ be a partition of~$[0,T]$ and let~$u \in \L^s([0,T],\R^m)$. When~$s=+\infty$, the inequality~$\Vert \II^\T(u) \Vert_{\L^\infty} \leq \Vert u \Vert_{\L^\infty}$ follows from~\eqref{eqdom}. When~$1 \leq s < +\infty$, we get from the H\"older inequality that
\begin{equation*}
 \left\Vert \dfrac{1}{t_{i+1} -t_i} \int_{t_i}^{t_{i+1}} u(\xi) d\xi \right\Vert_{\R^m}^s \!\!\!
  \leq \dfrac{1}{t_{i+1} -t_i} \int_{t_i}^{t_{i+1}} \Vert u(\xi) \Vert^s_{\R^m} d\xi 
\end{equation*}
for all~$i \in \{ 0,\ldots,N-1 \}$, and thus~$\Vert \II^\T(u) \Vert^s_{\L^s} \leq \Vert u \Vert^s_{\L^s}$. 
\end{proof}

The next lemma is instrumental in order to approximate with a~$\L^s$-norm (with any~$1 \leq s < +\infty$) any control~$u \in \L^\infty([0,T],\R^m)$ with piecewise constant controls.

\begin{lemma}\label{lemaverage2}
Let~$1 \leq s < +\infty$. Given any~$u \in \L^s([0,T],\R^m)$, we have~$\Vert \II^\T (u) - u \Vert_{\L^s} \rightarrow 0$ as~$\Vert \T \Vert \to 0$.
\end{lemma}

\begin{proof}
Let~$u \in \L^s([0,T],\U)$. Seeing~\eqref{eqdom} as a domination assumption and thanks to the Lebesgue dominated convergence theorem, we only need to prove that~$\II^\T(u)(\t) \rightarrow u(\t)$ as~$\Vert \T \Vert\rightarrow 0$ for a.e.\ $\t \in [0,T]$. For this purpose we set~$r(t) = \int_0^t u(\xi) \, d\xi$ for every~$t \in [0,T]$ and let~$\t \in [0,T)$ being a Lebesgue point such that~$r$ is derivable at~$\t$ with~$\dot{r}(\t) = u(\t)$. Given any~$\eps > 0$, there exists~$\delta > 0$ such that
$$ 
\left\Vert \dfrac{r(t)-r(\t)}{t - \t} - u(\t) \right\Vert_{\R^m} \leq \dfrac{\eps}{2} 
$$
for every~$t \in [\t-\delta,\t+\delta] \cap [0,T] \bs \{ \t \}$. Take~$\T$ a partition of~$[0,T]$ such that~$\Vert \T \Vert \leq \delta$. There exists~$i \in \{ 0, \ldots , N-1 \}$ such that $\t \in [t_i,t_{i+1})$. Then
\begin{multline*}
\Vert \II^\T(u) (\t) - u(\t) \Vert_{\R^m} = \left\Vert \dfrac{r(t_{i+1})-r(t_i)}{t_{i+1}-t_i} - u(\t) \right\Vert_{\R^m} \\
\leq \left\Vert \dfrac{r(t_{i+1})-r(\t)}{t_{i+1}-\t} - u(\t) \right\Vert_{\R^m} \left\vert \dfrac{t_{i+1}-\t}{t_{i+1}-t_i} \right\vert  \\[3pt]
+  \left\Vert \dfrac{r(\t)-r(t_i)}{\t-t_i} - u(\t) \right\Vert_{\R^m} \left\vert \dfrac{\t-t_i}{t_{i+1}-t_i} \right\vert \leq \eps, 
\end{multline*}
which concludes the proof.
\end{proof}

Our objective now is to prove that, when~$\U$ is convex, the averaging operators project any integrable function with values in~$\U$ onto a piecewise constant function with values in~$\U$. 

\begin{lemma}\label{lemaveraging}
Assume that~$\U$ is convex. If~$u \in \L^1([0,1],\U)$, then~$\int_0^1 u(\xi) d\xi \in \U$.
\end{lemma}

\begin{proof}
Let~$u \in \L^1([0,1],\U)$ and let us prove that~$\ttilde{u} \in \U$ where~$\ttilde{u}$ is defined by~$\ttilde{u} = \int_0^1 u(\xi) d\xi$. We first give a simpler argument when~$\U$ is furthermore assumed to be closed. In that context, by the Hilbert projection theorem, we have 
$$ \langle \ttilde{u} - \proj_\U ( \ttilde{u} ), u(\xi) - \proj_\U (\ttilde{u} ) \rangle_{\R^m} \leq 0   $$ 
for a.e.\ $\xi \in [0,1]$, where~$\proj_\U ( \ttilde{u} ) \in \U$ is the projection of~$\ttilde{u} $ onto~$\U$. Integrating the above inequality over~$[0,1]$ yields~$\Vert \ttilde{u} - \proj_\U ( \ttilde{u} ) \Vert_{\R^m}^2 \leq 0$ and thus~$\ttilde{u} = \proj_\U ( \ttilde{u} ) \in \U$.

Now we remove the closedness assumption made on~$\U$. Let us prove that~$\ttilde{u} \in \U$ by strong induction on the dimension~$d \in \N$ of the nonempty convex set~$\U$. If~$d=0$, the set~$\U$ is reduced to a singleton and the result is trivial. Now consider that~$d \geq 1$ and assume that the result is true at all steps from~$0$ to~$d-1$. By contradiction assume that~$\ttilde{u} \notin \U$. By separation, there exists~$\psi \in \R^m \backslash \{ 0_{\R^m} \}$ such that~$\langle \psi , \omega - \ttilde{u} \rangle_{\R^m} \leq 0 $ for all~$\omega \in \U$. We infer that the null integral~$ \int_0^1 \langle \psi , u(\xi) - \ttilde{u} \rangle_{\R^m} d\xi $ has a nonpositive integrand. Thus this integrand is zero almost everywhere on~$[0,1]$. Therefore~$u$ is with values in the convex set~$\U \cap ( \ttilde{u} + \psi^\bot  )$, where~$\psi^\bot$ stands for the standard hyperplane defined by orthogonality with the nonzero vector~$\psi$. Since~$\U \cap ( \ttilde{u} + \psi^\bot  )$ is a nonempty convex set of dimension strictly inferior than~$d$, thanks to our induction hypothesis we get that~$\ttilde{u} \in \U \cap ( \ttilde{u} + \psi^\bot  )$, which raises a contradiction.
\end{proof}

From Lemma~\ref{lemaveraging} and applying a simple affine change of variable in~\eqref{defaveraging}, we obtain the next proposition. 

\begin{proposition}\label{propaverage3}
Assume that~$\U$ is convex. If $u \in \L^1([0,T],\U)$, then~$\II^\T(u) \in \PC^\T([0,T],\U)$ for any partition~$\T$ of~$[0,T]$.
\end{proposition}

\subsection{Truncated end-point mapping and~$\L^s$-differential}\label{apptruncated}
For every~$M > 0$, we fix a mapping~$\Lambda^M : \R^n \times \R^m \to \R$ of class $\C^1$ satisfying
$$
\Lambda^M(x,u) = \left\lbrace \begin{array}{l}
1 \text{ if } (x,u) \in \BB_{\R^n}(0,2M) \times \BB_{\R^m}(0,2M), \\[3pt]
0 \text{ if } (x,u) \notin \B_{\R^n}(0,3M)  \times \B_{\R^m}(0,3M) .
\end{array} \right.
$$
Let~$M > 0$. When replacing the dynamics~$f$ in the control system~\eqref{contsys} by the \textit{truncated dynamics}~$f^M$, defined by~$ f^M (x,u,t) = \Lambda^M(x,u) f(x,u,t) $ for all~$(x,u,t) \in \R^n \times \R^m \times [0,T]$ we obtain a new control system that we denote by~($\mathrm{CS}^M$). The main difference is that, for any control~$u \in \L^1([0,T],\R^m)$ (even unbounded), there exists a trajectory~$x \in \AC ([0,T], \R^n)$, starting at~$x(0)=x^0$, such that~$\dot{x}(t) = f^M ( x(t),u(t),t )$ for a.e.\ $t \in [0,T]$. In that case the trajectory~$x$ is unique and will be denoted by~$x^M_u$. We now introduce, for any~$1 \leq s \leq +\infty$, the \textit{truncated end-point mapping}~$\E^M : \L^s([0,T],\R^m)\rightarrow\R^n$ defined by~$\E^M(u) = x^M_u(T)$ for all~$u \in \L^s([0,T],\R^m)$. Note that the next proposition, derived from standard techniques in ordinary differential equations theory, is true for any~$1 < s \leq +\infty$. The case~$s=1$ is discussed in Remark~\ref{remnotfrechet}.

\begin{proposition}\label{propinputoutputR}
Let~$1 < s \leq +\infty$ and~$M > 0$. The truncated end-point mapping~$\E^M:\L^s([0,T],\R^m)\rightarrow\R^n$ is of class~$\C^1$ and its Fr\'echet differential is given by
\begin{equation}\label{eqdiffM}
\D \E^M (u) \cdot v = w^{u,M}_v(T) 
\end{equation}
for all $u$, $v \in \L^s([0,T],\R^m)$, where $w^{u,M}_v \in \AC([0,T],\R^n)$ is the unique solution to
\begin{equation*}
\left\lbrace
\begin{array}{l}
\dot{w}(t) =  \nabla_x f^M(x^M_u(t),u(t),t) w(t) \\
\qquad\qquad + \nabla_u f^M(x^M_u(t),u(t),t) v(t), \quad \text{a.e.\ } t \in [0,T] , \\
w(0) =  0_{\R^n} .
\end{array}
\right.
\end{equation*} 
\end{proposition}

\begin{remark}\label{remtruncated}
Let~$1 < s \leq +\infty$. For a given control~$u \in \UU$, note that~$\D \E (u)$ given in~\eqref{eq1} admits a natural extension (still denoted by)~$\D \E (u):\L^s([0,T],\R^m)\rightarrow\R^n$. The nontruncated setting is related to the truncated one as follows:
\begin{enumerate}[label=\rm{(\roman*)}]
\item Let~$u \in \mathcal{U}$ and~$M > 0$ be such that~$\Vert x_u \Vert_{\C} \leq M$ and~$\Vert u \Vert_{\L^\infty} \leq M$. Then~$x^M_u = x_u$ and~$\D \E^M (u) = \D \E (u)$ when considering the above extension of~$\D \E (u)$.
\item Let~$u \in \L^\infty([0,T],\R^m)$. If there exists~$M > 0$ such that~$\Vert x^M_u \Vert_{\C} \leq M$ and~$\Vert u \Vert_{\L^\infty} \leq M$, then~$u \in \mathcal{U}$ and~$x_u = x^M_u$.
\end{enumerate}
\end{remark}

\begin{remark}\label{remnotfrechet}
Let~$M > 0$. In the case~$s=1$, it can be proved that the truncated end-point mapping~$\E^M:\L^1([0,T],\R^m)\rightarrow\R^n$ is Gateaux-differentiable and its Gateaux differential is given by~\eqref{eqdiffM}. However it is not Fr\'echet-differentiable (and thus not of class~$\C^1$) in general, as shown in the next example.
\end{remark}

\begin{example}
Take~$T=n=m=1$, $\U = \R$ and~$f(x,u,t) = u^2$ for all~$(x,u,t) \in \R \times \R \times [0,T]$. Consider the starting point~$x^0 = 0$ and the constant control~$u \equiv 0$. In that context, with~$M=s=1$, it is clear that~$x^M_u \equiv 0$ and that the Gateaux differential~$\D^{\mathrm{G}} \E^M(u) :\L^1([0,T],\R^m)\rightarrow\R^n$ of the truncated end-point mapping~$\E^M:\L^1([0,T],\R^m)\rightarrow\R^n$ at~$u$, given by the expression~\eqref{eqdiffM}, is null. Now, taking the needle-like variation~$u^\alpha_{(0,1)}$, as defined in~\eqref{eq_singleneedle}, associated with the pair~$(0,1) \in \mathcal{L}( f_u) \times \U $, we obtain that
$$ \lim\limits_{\alpha \to 0^+} \dfrac{ \E^M( u + u^\alpha_{(0,1)} ) - \E^M (u) - \D^{\mathrm{G}} \E^M(u) \cdot u^\alpha_{(0,1)} }{ \Vert u^\alpha_{(0,1)} \Vert_{\L^1} } = 1. $$
Therefore~$\E^M:\L^1([0,T],\R^m)\rightarrow\R^n$ is not Fr\'echet-differentiable at~$u$.
\end{example}


\bibliographystyle{plain}
\bibliography{bibIEEEbourdintrelat}

%
%


\begin{IEEEbiography}{Lo\"ic Bourdin}
was born in 1986. He received his PhD degree in Applied Mathematics from the University of Pau (France) in 2013. Since 2014, he is associate professor at the University of Limoges (France) and since 2016, he is member of SMAI-MODE group. His mathematical interests are control theory, optimal control, shape optimization problems and nonsmooth analysis.
\end{IEEEbiography}

\begin{IEEEbiography}{Emmanuel Tr\'elat}
was born in 1974. He is full professor at Sorbonne Universit\'e and director of Laboratoire Jacques-Louis Lions. He is the Editor in Chief of the journal \emph{ESAIM: Control Optim. Calc. Var.}, and is Associate Editor of several other journals. He has been awarded the \emph{SIAM Outstanding Paper Prize} (2006), \emph{Maurice Audin Prize} (2010), \emph{Felix Klein Prize} (European Math. Society, 2012), \emph{Blaise Pascal Prize} (french Academy of Science, 2014), \emph{Big Prize Madame Victor Noury} (french Academy of Science, 2016). His research interests range over control theory in finite and infinite dimension, optimal control, stabilization, geometry, numerical analysis.
\end{IEEEbiography}

\end{document}